\documentclass[a4paper,10pt]{article}
\usepackage[english]{babel}
\usepackage[colorlinks,citecolor=green,linkcolor=red]{hyperref}
\usepackage{amsthm}
\usepackage{xcolor}
\usepackage{mathrsfs}
\usepackage{amsmath}
\usepackage{mathtools}
\usepackage{amsfonts}
\usepackage{amssymb}
\usepackage{bm}
\usepackage{physics}
\usepackage{enumitem}
\usepackage{a4wide}
\usepackage{cleveref}
\usepackage{esint}
\usepackage{nicefrac}
\usepackage{yfonts}
\usepackage{dirtytalk}

\numberwithin{equation}{section}

\newtheorem{thm}{Theorem}[section]
\newtheorem*{thm*}{Theorem}
\newtheorem{lem}[thm]{Lemma}
\newtheorem{prop}[thm]{Proposition}
\newtheorem{cor}[thm]{Corollary}

\theoremstyle{definition}
\newtheorem{defn}[thm]{Definition}

\theoremstyle{remark}
\newtheorem{rem}[thm]{Remark}

\newcommand{\fr}{\penalty-20\null\hfill$\blacksquare$}

\newcommand{\defeq}{\vcentcolon=}

\def\Xint#1{\mathchoice
	{\XXint\displaystyle\textstyle{#1}}%
	{\XXint\textstyle\scriptstyle{#1}}%
	{\XXint\scriptstyle\scriptscriptstyle{#1}}%
	{\XXint\scriptscriptstyle\scriptscriptstyle{#1}}%
	\!\int}
\def\XXint#1#2#3{{\setbox0=\hbox{$#1{#2#3}{\int}$}
		\vcenter{\hbox{$#2#3$}}\kern-.5\wd0}}
\def\dashint{\Xint-}

\newcommand{\mres}{\mathbin{\vrule height 1.6ex depth 0pt width 0.13ex\vrule height 0.13ex depth 0pt width 1.3ex}}

\newcommand{\RCD}{{\mathrm {RCD}}}

\newcommand{\XX}{{\mathsf{X}}}
\newcommand{\YY}{{\mathsf{Y}}}
\newcommand{\dist}{{\mathsf{d}}}
\newcommand{\mass}{{\mathsf{m}}}

\newcommand{\LIP}{{\mathrm {LIP}}}
\newcommand{\BV}{{\mathrm {BV}}}

\newcommand{\lip}{{\mathrm {lip}}}

\newcommand{\DIFF}{{\mathrm{D}}}

\newcommand{\QQ}{{\mathbb{Q}}}
\newcommand{\RR}{{\mathbb{R}}}
\newcommand{\ZZ}{{\mathbb{Z}}}
\newcommand{\NN}{{\mathbb{N}}}
\newcommand{\HH}{{\mathcal{H}}}
\renewcommand{\SS}{{\mathcal{S}}}

\newcommand{\per}{{\mathrm {Per}}}
\newcommand{\nchi}{{\raise.3ex\hbox{$\chi$}}}

\let\phi\varphi
\let\epsilon\varepsilon

\DeclareMathOperator*{\essinf}{ess\,inf}
\title{Maps of bounded variation \\from PI spaces to metric spaces}

\begin{document}
\author{Camillo Brena\footnote{
{camillo.brena@sns.it}, Scuola Normale Superiore, Piazza dei Cavalieri 7, 56126 Pisa, Italy.},
\,Francesco Nobili\footnote{
{francesco.f.nobili@jyu.fi}, Department of Mathematics and Statistics, P.O.\ Box 35 (MaD), FI-40014 University of Jyv\"askyl\"a, Finland.},
\,and Enrico Pasqualetto\footnote{
{enrico.e.pasqualetto@jyu.fi}, Department of Mathematics and Statistics, P.O.\ Box 35 (MaD), FI-40014 University of Jyv\"askyl\"a, Finland.}}
\date{
}
\maketitle

\begin{abstract}
      We study maps of bounded variation defined on a metric measure space and valued into a metric space. Assuming the source space to satisfy a doubling and Poincar\'e property, we produce a well-behaved relaxation theory via approximation by simple maps. Moreover, several equivalent characterizations are given, including a notion in weak duality with test plans.
\end{abstract}

\tableofcontents

\section{Introduction}
In the classical Euclidean setting, there are many equivalent characterizations of a real-valued function of bounded variation in addition to the standard distributional notion. To mention two relevant ones, the existence of equi-bounded  (in the energy sense) smooth approximations and that one-dimensional restrictions are of bounded variation and satisfy integral energy bounds. We refer to \cite{AmbrosioFuscoPallarabook} for this classical theory.

It was later understood that the two just mentioned approaches, suitably adapted, make perfect sense on an arbitrary metric measure space $(\XX,\dist,\mass)$, i.e.\ a metric space $(\XX,\dist)$ equipped with a reference measure $\mass$. The challenge being to re-interpret these characterizations in metric terms and cook up new strategies to prove their equivalence, ranging from measure theory to gradient flows and optimal transportation. The picture is rather complete by now: \cite{MirandaJr03} introduced a definition of real-valued functions of bounded variation
on a metric measure space via relaxation; later on, \cite{AmbrosioDiMarino14} proposed another definition in the spirit of
`checking the behaviour of functions along curves' (formulated in terms of the so-called test plans) and proved its equivalence
with Miranda's approach; in \cite{DiMarino14,DiMarinoPhD} a new definition via integration-by-parts formulas (using derivations), as well
as its equivalence with the previous two approaches was studied; finally, in \cite{Martio16,Martio16-2} the concept of
AM-modulus and the relative curvewise notion of BV-function were introduced; the equivalence with the previous three approaches
was first partially proved in \cite{DCEBKS19}, then the full equivalence result was obtained in \cite{NobiliPasqualettoSchultz21}, thus closing the circle.

If one assumes some sort of regularity of the underlying space, such as a doubling \& Poincar\'e property (we refer to those as PI spaces, see Section \ref{sec:prelim}), then it is possible to further study the space of BV-functions (motivated by the analogous Sobolev theory referring, e.g., to \cite{HajlaszKoskela00,HeinonenKoskelaShanmugalingam15, Bjorn-Bjorn11} and references therein).  In this situation, the picture looks much closer to the Euclidean one and finer properties can be deduced, see \cite{Ambrosio2001,Ambrosio2002,AmbrosioMirandaPallara04,KinnunenKorteShanmugaligam14} and  \cite{LTpointwise}  that are particularly relevant for us. 
\medskip

The aim of this note is to study \emph{maps of bounded variation}. These naturally arise in a variety of situations, such as $\Gamma$-convergence problems in the calculus of variations (see, e.g., \cite{ModicaMortola77,Baldo90,ambmetric}, in the study of cartesian currents (see, e.g., \cite{GiaquintaModicaSoulcek98I,GiaquintaModicaSoulcek98II}), liftings of manifold-valued maps (see, e.g., \cite{GiaquintaModicaSoulcek98I,DavilaIgnat03,Ignat05,Canevari17,IgnatLamy19,CanevariOrlandi20}) and also bi-Lipschitz nonembeddability \cite{CheegerKleiner10}.

We focus on maps valued into an arbitrary separable metric space $(\YY,\rho)$:
\[
u\colon \XX\to\YY.
\]
It is important to stress here that the work \cite{MirandaJr03} actually deals with Banach-valued maps of bounded variation. However, if one wants -- as in this note -- to replace a Banach target with a \emph{non-linear} structure, such as a Riemannian manifold or more generally a metric space, then the situation becomes more delicate.

A first attempt would be to embed the target space in a linear space (as done in \cite{GiaquintaMucci06,GiaquintaMucci06I,GiaquintaMucci07,GiaquintaMucci_Errata}) for smooth manifold target spaces via Nash embedding. This can be in principle done also in the generality of this manuscript, considering isometric embeddings of separable metric targets into a suitable Banach space. However, we are interested in studying a theory that is \emph{intrisic} rather than developing an \emph{extrinsic} theory depending on a chosen embedding. Moreover, we found no real advantage in exploiting the extrinsic procedure.

\medskip

Starting from \cite{Reshetnyak97,Reshetnyak04}, 
 after \cite{ambmetric}, it has been understood that an intrinsic way to discuss metric-valued calculus is by looking at post-compositions with Lipschitz functions of the target. In particular, given a (indeed, any) notion of $\BV(\XX)$ and total variation measure $|\DIFF f|$ for real-valued functions $f \in L^1(\XX)$, we can declare for $u \in L^1(\XX,\YY)$ the following:
\begin{equation}
u \in \BV(\XX,\YY)\quad \text{provided} \quad \varphi \circ u \in \BV(\XX)\text{ and } |\DIFF (\varphi\circ u)|\le \mu,
\label{eq:Reshetnyak intro}
\end{equation}
for some finite Borel measure $\mu$ and for all $1$-Lipschitz functions $\varphi \colon \YY\to \RR$ (see Definition \ref{def:Du Resh}). Then, set $|\DIFF u|$ as the minimal measure $\mu$ for the above to hold. This approach has the advantage of being the most general possible, as it requires no regularity of $\XX$ nor of $\YY$. 
\medskip

More sophisticated approaches that are particularly relevant for this work have been investigated, at different levels of generality, in \cite{ambmetric,KorevaarSchoen93}. In \cite{ambmetric}, a notion of maps of bounded variation defined on the Euclidean space and valued into a metric space has been given via relaxation with simple maps. Then, equivalent characterizations and further properties are investigated. Mimicking \cite{ambmetric}, we can define
\begin{equation}
    \label{eq:BV* intro}
u \in \BV^*(\XX,\YY)\quad\text{provided} \quad \exists \, u_n \colon \XX\to \YY \quad\text{locally simple s.t.}\quad \begin{array}{l}
      u_n \to u \text{ in }L^1_{\rm loc},\\
     \sup_n|\DIFF u_n|(\XX)<\infty,
\end{array}
\end{equation}
where a map is called `simple' provided it is of finite range in $\YY$ and where $|\DIFF u_n|$ is defined via \eqref{eq:Reshetnyak intro}. We point out that, if the target space $\YY$ is totally disconnected, then asking for the existence of Lipschitz approximations of a map $u$ (as for the real-valued case) is not an effective requirement. On the other hand, maps attaining locally a finite number of values are suitable for approximations even when dealing with arbitrary metric targets. As in the previous case, the definition above comes with a notion of total variation on open sets, that we call $V_u$, see \eqref{cdsacdas}, which turns out to be induced by a Borel measure (still denoted by $V_u$) whenever the source space has a doubling and Poincar\'e property, see Theorem \ref{thm:main result 1} extending \cite{ambmetric}.
\medskip 

In \cite{KorevaarSchoen93}, directional and total energies of Sobolev and bounded variation maps have been studied for maps defined on manifolds with Ricci curvature lower bounds and valued in an arbitrary metric space (see also \cite{LogaritschSpadaro12,IgnatLamy19}). The work \cite{KorevaarSchoen93} reveals that it is possible to develop, even when both source and target spaces are non-linear, a theory of maps of bounded variation that ultimately looks at one-dimensional restrictions along flow-curves driven by Lipschitz vector fields. For what concerns Sobolev calculus, recently this theory has been generalized to non-smooth source spaces with synthetic Ricci lower bounds in \cite{GigliTyulenev21Dir,GT21}. These works motivate our investigation around BV-maps but, with the aim of covering more general source spaces and treat less regular functions, our results will be closer in spirit to the weak-BV approach for functions developed in \cite{AmbrosioDiMarino14}.  To this aim, we shall enforce the requirement of `1-dimensional restrictions are BV', with the concept of test plans \cite{AmbrosioGigliSavare11,AmbrosioGigliSavaredensity}. Recall that an $\infty$-test plan is a probability measure on the space of continuous paths, i.e.\ $\pi \in \mathcal{P}(C([0,1],\XX))$, that has bounded compression and is concentrated on equi-Lipschitz curves; denote ${\rm Comp}(\pi)$ and ${\rm Lip}(\pi)$ its compression and Lipschitz constant, respectively (see Section \ref{sec:prelim} for the precise definition of test plans). We can then define
\begin{equation}
    u \in \BV_w(\XX,\YY)\quad \text{provided}   \quad \begin{array}{l}
        \text{for any \(\infty\)-test plan \(\pi\) on \(\XX\), }\,u\circ \gamma \text{ is BV for }\pi\text{-a.e.\ }\gamma, \\
         \int |\DIFF (u\circ \gamma)|(\gamma^{-1}(B))\,\dd\pi(\gamma)\le {\mathrm {Comp}}(\pi) {\mathrm {Lip}}(\pi)\mu(B),
    \end{array}\label{eq:BV weak intro}
\end{equation}
for some finite measure $\mu$ (independent of $\pi$), for any $B\subseteq \XX$ Borel  (see 
 Definition \ref{def:Du weak} for the above). Set $|\DIFF u|_w$ the smallest measure $\mu$.
\medskip

Before discussing the main results of this note, we point out that when the target space is $\YY=\RR$ (but the source space $\XX$ is an arbitrary metric measure space), all the above notions do coincide with the one introduced by \cite{MirandaJr03}. More precisely, $|\DIFF u|$, $V_u$ and $|\DIFF u|_w$ coincide exactly with the notion of total variation defined in \cite{MirandaJr03}. This follows easily from standard arguments, the equivalence result of \cite{AmbrosioDiMarino14} (to deal with $|\DIFF u|_w$) and the approximation scheme in \cite[proof of Proposition 4.2]{MirandaJr03} (to deal with $V_u$).

\subsection*{Statement of results}
To be more precise, here a PI space is a uniformly locally doubling metric measure space supporting a local $(1,1)$-Poincar\'e inequality (see Section \ref{sec:prelim}). The main goal of this work is:
\begin{center}
    to study \eqref{eq:Reshetnyak intro}-\eqref{eq:BV* intro}-\eqref{eq:BV weak intro} and show their equivalence when $\XX$ is a PI space.
\end{center}
Even though all these approaches make perfect sense on an arbitrary metric measure source space, {we require $\XX$ to be PI to produce} a well-behaved relaxation theory in \eqref{eq:BV* intro} as well as to address the proof of the harder implications among these three approaches. { It is unclear whether, in full generality, all these approaches coincide.}

Indeed, while \eqref{eq:Reshetnyak intro}-\eqref{eq:BV weak intro} come automatically with a notion of total variation measure, an important step in our investigation is to show that there is an underlying measure which can be obtained localizing on open subsets the relaxation \eqref{eq:BV* intro}. For every open set $U\subseteq \XX$ define
\begin{equation}\label{cdsacdas}
V_u(U) \coloneqq \inf\Big\{ \liminf_{n\to\infty}|\DIFF u_n|(U)\;\Big|\; u_n \colon U\to\YY\text{ locally simple}, u_n \to u \text{ in }L^1_{\rm loc} \Big\}.
\end{equation}

Our first main result is the following:
\begin{thm}\label{thm:main result 1}
    Let $(\XX,\dist,\mass)$ be a PI space, let $(\YY,\rho)$ be a separable metric space and let $u \in L^1_{\rm loc}(\XX,\YY)$. If $V_u(\XX)<\infty$, then $V_u$ is the restriction to open sets of a finite Borel measure.
\end{thm}
This will be proved in Theorem \ref{thm:V is a measure}, thus generalizing \cite{ambmetric} to the current framework. To prove it, we perform a fine investigation around simple maps to characterize their total variations (up to structural constants, see Lemma \ref{cdscascs}). Then, a delicate joint property of locally simple maps can be proved (see Lemma \ref{lem:joint}) to make a general strategy work and prove that the Carath\'eodory extension of $V_u$ is indeed a measure.

We can then present our second main result (we will consider pointed metric targets).
\begin{thm}\label{thm:main result 2}
     Let $(\XX,\dist,\mass)$ be a PI space and let $(\YY,\rho)$ be a separable metric space. Then:
     \[
     \BV(\XX,\YY) = \BV^*(\XX,\YY) = \BV_w(\XX,\YY).
     \]
     Moreover, if $u \in \BV(\XX,\YY)$, we have for some $C>0$ depending only on the PI-parameters:   
     \[
        |\DIFF u |\le V_u \le C |\DIFF u|, \qquad |\DIFF u|\le |\DIFF u|_w \le C|\DIFF u|.
     \]
\end{thm}
This follows as a combination of several inclusions proved in Section \ref{sec:equivalence}. To prove the above, we achieve in Proposition \ref{thm:Equivalent BVmaps PI} two other characterizations of independent interest:
\begin{itemize}
    \item[i)] A map is in $\BV(\XX,\YY)$ if and only if it satisfies a maximal-type estimate in the spirit of \cite{Haj03,HajNew} and, for the BV-case, \cite{LTpointwise}.
    \item[ii)] A map is in $\BV(\XX,\YY)$ if and only if it satisfies a suitable metric-type Poincar\'e inequality.
\end{itemize}

\medskip

We conclude this Introduction by discussing the case of more regular source/target metric spaces. Not surprisingly, we expect that assuming synthetic Ricci curvature lower bounds at the source level makes it possible to investigate finer properties in our theory thanks to the deep understanding of the co-dimension one structure of the so-called $\RCD$-spaces (see \cite{AmbrosioICM} for a thorough discussion)  achieved in \cite{AMbrosioBrueSemola19,BruPasSem19,BruPasSem21-constantcodimension}. Nowadays, there is a solid fine theory of BV-functions defined on $\RCD$-spaces, resembling the classical Euclidean framework, \cite{BrenaGigli22_Calculus,BrenaGigli22_Local,AntonelliBrenaPasqualetto22_Rank,AntonelliBrenaPasqualetto22_Graph,BrenaPasqualettoPinamonti22}. We postpone to a future investigation the study of fine properties of BV-maps in such setting.

Finally, we mention a possible research direction assuming also non-positive sectional curvature bounds in the sense of Alexandrov of the target space (see, e.g., \cite{BH99,BBI01} and references therein). A natural investigation would be to revisit the Korevaar-Schoen theory \cite{KorevaarSchoen93} for maps of bounded variation with $\RCD$-source space and ${\sf CAT}(0)$-target as done recently in \cite{GigliTyulenev21Dir,GT21} for the case of Sobolev maps. Thanks to the (universal) infinitesimal Hilbertianity of the target \cite{DiMarinoGigliPasqualettoSoultanis21}, one can then study `parallelogram identities'  for the pairing of $|\DIFF u|$ with a regular vector field inducing a weak type of flow \cite{AmbrosioTrevisan14}. 

\medskip

\noindent\textbf{Acknowledgements}.  Part of this work was carried out at the Fields Institute during the Thematic Program on Nonsmooth Riemannian and Lorentzian Geometry, Toronto 2022. The authors gratefully acknowledge the warm hospitality and the stimulating atmosphere. Moreover, we thank the referee for useful suggestions and a careful reading.

C.B.\ is supported by the PRIN MIUR project “Gradient flows, Optimal Transport and Metric Measure Structures”.

F.N.\ is supported by the Academy of Finland projects \emph{Geometric Aspects of Sobolev
Space Theory}, Grant No.\ 314789 and \emph{Incidences on Fractals}, Grant No.\ 321896. 

E.P.\ thanks Davide Carazzato for the useful discussions on the results of Section \ref{ss:var_simple}.

F.N.\ and E.P.\ also acknowledge the support of the Academy of Finland  Grant No.\ 323960.

\section{Preliminaries}\label{sec:prelim}

\subsection{Notation}
For our purposes, a metric measure space is a triplet \((\XX,\dist,\mass)\), where
\[\begin{split}
(\XX,\dist)&\quad\text{ is a complete, separable metric space},\\
\mass\geq 0&\quad\text{ is a Borel measure on }(\XX,\dist)\text{ that is finite on bounded subsets}.
\end{split}\]
Denote by \(\mathscr{P}(\XX)\) the family of all Borel probability measures on \((\XX,\dist)\), equipped with the weak topology in duality with continuous and bounded functions \(C_b(\XX)\). Given $\varphi \colon \XX \to \YY$ Borel, for a metric target $(\YY,\rho)$, and $\mu \in \mathscr{P}(\XX)$, the \emph{pushforward} measure of $\mu$ via $\varphi$ is defined as $\varphi_* \mu (B) \coloneqq \mu(\varphi^{-1}(B))$, for all $B\subset\YY$ Borel.

Given $\varnothing \neq \Omega\subset \XX$ open, we denote $L^1_{\rm loc}(\Omega)$ the space of Borel locally integrable functions $f \colon \Omega \to \RR$ quotiented up to $\mass$-a.e.\ equality in $\Omega$. By locally integrable, we mean that every point has an open neighbourhood $U$ (depending on $x$) so that $f \in L^1(U)$.

For a metric space $(\YY,\rho)$, we denote $L^1_{\rm loc}(\Omega,\YY)$ the space of maps $u \colon \Omega \to \YY$ with separable range defined up to $\mass$-a.e.\ equality in $\Omega$ so that $\rho(u,\bar y) \in L^1_{\rm loc}(\Omega)$ for one point (hence all points)  $\bar y \in \YY$. If also $\rho(u,\bar y) \in L^1(\Omega)$, then we say that $u \in L^1(\Omega,\YY_{\bar y})$. In this situation, it will be convenient to consider a \emph{pointed} metric space $(\YY,\rho,\bar y)$. When either $\mass(\Omega)<\infty$ or we have a standard target $\YY=\RR$ (or a Banach space), the dependence on $\bar y$ will be dropped. We shall tacitly drop the point $\bar y$ in all the forthcoming notations in these two situations.
Notice that in order to impose integrability assumptions on a map, we ask , as customary, that its image lies on a separable subset of the target. For the sake of simplicity, throughout this manuscript we will assume that the whole target is separable. This causes no loss of generality in our results.

We denote by $\LIP(\Omega)$ the set of Lipschitz function $f \colon \Omega \to \RR$ and by $\LIP_{\rm loc}(\Omega)$ the set of functions so that for every point there exists a neighbourhood where they are Lipschitz. By $\lip(f)$, we denote the local lipshitz constant of $f \colon \Omega \to \RR$ defined as $\limsup_{y\to x}\frac{|f(y)-f(x)|}{\dist(x,y)}$ if $x \in \Omega$ is not isolated and  $0$ otherwise.

In this manuscript, we will mainly work in proper metric measure spaces, i.e.\ spaces in which bounded and closed sets are compact. We say that $U$ open is well-contained inside $\Omega$, and write $U \Subset \Omega $ provided $U$ is bounded and so that $\dist(U,\Omega^c)>0$. When $\Omega =\XX$, we only ask for $U$ to be bounded. When the underlying space is proper, then any such $U$ is compact in $\Omega$ and we single out the following equivalences for later use:
\[\begin{split}
f \in L^1_{\rm loc}(\Omega) \qquad &\Leftrightarrow\qquad f \in L^1(U), \qquad \forall U \Subset \Omega,\\
f \in \LIP_{\rm loc}(\Omega) \qquad &\Leftrightarrow \qquad         f \in \LIP(U), \qquad \forall U \Subset \Omega. \nonumber
\end{split}\]
Finally, throughout this manuscript, we shall indicate with $C>0$ a general constant that during the estimates might change from line to line, without notice. Even though {its value is explicit}, we shall only keep track of its dependence on suitable parameters of the space, rather than writing its value.

\subsection{Metric valued curves and test plans}
Denote by $C([0,1],\XX)$ the set of continuous and $\XX$-valued curves defined on the unit interval and equipped with the sup distance. We define the evaluation map ${\sf e} \colon [0,1]\times C([0,1],\XX) \to {\sf e}(\gamma,t) \coloneqq \gamma_t$ and observe that it is continuous. For $t \in [0,1]$, denote $\gamma \mapsto {\sf e}_t(\gamma) \coloneqq {\sf e}(\gamma,t) $ the evaluation map at time $t$ and, for $0\le t<s\le 1$, also the restriction operation ${\rm restr}_t^s \colon C([0,1],\XX) \to C([0,1],\XX)$ defined by $ {\rm restr}_t^s (\gamma) \coloneqq \gamma_{(1-\cdot)t +\cdot s}$.  

We denote $\LIP([0,1],\XX)$ the subset of Lipschitz curves with values in $\XX$. It is well known that, given \( \gamma\in\LIP([0,1],\XX)\) there exists
\(|\dot\gamma_t|\coloneqq\lim_{h\to 0}\dist(\gamma_{t+h},\gamma_t)/|h|\) for a.e.\ \(t\in[0,1]\) with  $|\dot\gamma| \in L^\infty(0,1)$ (see \cite[Theorem 1.1.2]{AmbrosioGigliSavare08}).

Let us define for convenience the metric speed functional
\({\sf ms}\colon C([0,1],\XX)\times[0,1]\to[0,+\infty]\):
\[
{\sf ms}(\gamma,t)\coloneqq|\dot\gamma_t|,
\quad\text{ whenever }\gamma\in \LIP([0,1],\XX)\text{ and }
\exists\lim_{h\to 0}\frac{\dist(\gamma_{t+h},\gamma_t)}{|h|},
\]
and \({\sf ms}(\gamma,t)\coloneqq+\infty\) otherwise. We are now ready to give the definition of $\infty$-test plans.
\begin{defn}
Let \((\XX,\dist,\mass)\) be a metric measure space. A measure \(\pi\in\mathscr P\big(C([0,1],\XX)\big)\) is an $\infty$-test plan, provided
\[\begin{split}
&\exists\,{\rm C}>0:\quad\forall t\in[0,1],\quad({\sf e}_t)_*\pi\leq{\rm C}\mass,\\
&\|{\sf ms}\|_{L^\infty(\pi\otimes\mathcal L_1)}<+\infty.
\end{split}\]
The minimal constant $C>0$ for which the above holds is called compression constant and denoted by \({\rm Comp}(\pi)\). The energy of $\pi$ is then denoted ${\rm Lip}(\pi) \coloneqq \|{\sf ms}\|_{L^\infty(\pi\otimes\mathcal L_1)}$.
\end{defn}
We point out that the notation of energy is motivated by the following observation. If $\pi$ is an $\infty$-test plan, then it is concentrated on ${\rm Lip}(\pi)$-Lipschitz curves (see, e.g.,  \cite[Remark 2.2]{NobiliPasqualettoSchultz21}). 
In this note, we shall also deal with BV-curves with values in a metric space. 
\begin{defn}[BV-curve]\label{def:BV curve}
Let $(\YY,\rho)$ be a separable metric space and let $\beta \colon [0,1]\to\YY$ be Borel. Given $I\subseteq [0,1]$ define the \emph{variation} 
\[ {\rm Var}_\beta(I):=  \sup \sum_i \rho(\beta_{t_{i+1}},\beta_{t_i}) <\infty,\]
where the sup is taken over all finite partitions $(t_i)$ of $I$. 

We say that $\gamma \in L^1([0,1],\YY)$ is a BV-curve, and we write $\gamma \in \BV([0,1],\YY)$, provided
\[
|\DIFF \gamma|([0,1]) \coloneqq \inf \{ {\rm Var}_\beta([0,1]) \colon \beta \text{ Borel representative of }\gamma \}<\infty.
\]  
\end{defn}

It can be deduced that there is a well-defined Borel measure which, by abuse of notation, we still denote by $|\DIFF \gamma|$, that localizes to Borel sets the above construction. Moreover, it holds
\begin{equation}
    \gamma_n \to \gamma \text{ in }L^1([0,1],\YY) \qquad \Rightarrow \qquad |\DIFF \gamma |(I) \le \liminf_{n\to\infty}|\DIFF \gamma_n| (I)\label{eq:BV curve lsc}
\end{equation}
for all $I\subseteq [0,1]$ open. We also have the following characterization:
\begin{equation}
\gamma \in \BV([0,1],\YY)\quad \text{if and only if} \quad \varphi \circ \gamma \in \BV(0,1)\text{ and } |\DIFF (\varphi\circ \gamma)|\le \mu,
\label{eq:BV curve postcomp}
\end{equation}
for some finite Borel measure $\mu$ and for all $1$-Lipschitz functions $\varphi \colon \YY\to\RR$.

Finally, $|\DIFF \gamma|$ coincides with the smallest measure $\mu$ for the above to hold. See, e.g., \cite[Theorem 2.17]{AbediLiSchultz22}, for a list of equivalent characterizations where all these claims can be deduced.
\subsection{BV-functions and sets of finite perimeter in PI spaces}

We recall from \cite{MirandaJr03,AmbrosioDiMarino14} the notion of a function of bounded variation.
\begin{defn}
Let $(\XX,\dist,\mass)$ be a metric measure space, let $\varnothing \neq \Omega\subseteq\XX$ be open and let $f\in L^1_{\rm loc}(\Omega)$. Given any open set $U\subseteq\Omega$,
we define the \emph{total variation} of $f$ on $U$ as
\[
|\DIFF f|(U)\coloneqq {\inf}\Big\{\liminf_{n\to\infty} \int_U \lip(f_n)\,\dd \mass \colon (f_n)_n\subset \LIP_{\rm loc}(U), f_n\to f \text{ in }L^1_{\rm loc}(U)\Big\}.
\]
Then $f \in L^1_{\rm loc}(\Omega)$ is of \emph{locally bounded variation}, writing $f \in \BV_{\rm loc}(\Omega)$, provided for every point there is a neighborhood where $|\DIFF f|$ is finite. If $f \in L^1(\Omega)$ and $|\DIFF f|(\Omega)<\infty$, then $f$ is of bounded variation and we simply write $f \in \BV(\Omega)$.
\end{defn}
By a diagonal argument, it can be readily checked that the energy $|\DIFF f|(\Omega)$ is lower semicontinuous with respect to $L^1_{\rm loc}$-convergence, i.e.\ for every $U\subseteq \Omega$ open it holds
\[
    f_n \to f \text{ in }L^1_{\rm loc}(\Omega) \qquad \Rightarrow\qquad  |\DIFF f|(U) \le \liminf_{n\to\infty} |\DIFF f_n|(U).
\]
By a standard Carath\'eodory construction, given any $f \in L^1_{\rm loc}(\Omega)$ the {set-function} $|\DIFF f|$ defined on open sets can be extended to a well-defined Borel measure. In the case of $f = \nchi_E$, with $E\subseteq\Omega$ Borel, we say that $E$ is a \emph{set of finite perimeter} in \(\Omega\)
if $|\DIFF \nchi_E|(\Omega)<\infty$. The \emph{perimeter measure} is then given by $\per(E,B)\coloneqq |\DIFF \nchi_E|(B)$ for every $B\subseteq\Omega$ Borel.

We recall the following coarea formula that holds on general metric measure spaces, following verbatim the proof of \cite{MirandaJr03}, see \cite{AmbrosioDiMarino14}.
\begin{thm}
Let $(\XX,\dist,\mass)$ be a metric measure space and let $\varnothing \neq \Omega\subseteq\XX$ be open. Fix $f \in \BV_{\rm loc}(\Omega)$ and $E\subseteq\Omega$ Borel.
Then, the function $\RR\ni t\mapsto \per(\{f>t\},E) \in [0,\infty]$ is Borel measurable and it holds that
\[
|\DIFF f|(E) = \int_\RR \per(\{f>t\},E)\,\dd t.
\]
\end{thm}
The following lemma is useful to deduce the $\BV$ membership starting from the $\BV$ membership on sets of a covering. Its proof is essentially contained in the one of \cite[Lemma 5.2]{AmbrosioDiMarino14}, we give the details.
\begin{lem}\label{footnote9}
Let $(\XX,\dist,\mass)$ be a metric measure space, let $\varnothing\neq \Omega\subseteq\XX$ be open and let $f\in L^1(\Omega)$. Assume that there exists a sequence of open sets $(A_i)_i$ such that $\bigcup_i A_i=\Omega$ and, for every $i$, $f\in \BV(A_i)$ and moreover $\bigvee_i |\DIFF f|\mres A_i$ is a finite measure. Then
$f\in \BV(\Omega)$.
\end{lem}
\begin{proof}
    We first show that if $A,B\subseteq\XX$ are open, $f\in \BV(A)$ and $f\in \BV(B)$, then $f\in \BV(A\cup B)$. Let $(A_h)_h$ be a sequence of open sets such that $A_h\nearrow A$ and $A_h\Subset A$ for every $h$, and let similarly $(B_h)_h$. Now we can follow the proof of item $ii)$ of the proof of \cite[Lemma 5.2]{AmbrosioDiMarino14} to see that $f\in \BV(A_h\cup B_h)$ with $|\DIFF f|\mres (A_h\cup B_h)\le |\DIFF f|\mres A\vee |\DIFF f|\mres B$. Hence, by (the proof of) item $i)$ of the proof of \cite[Lemma 5.2]{AmbrosioDiMarino14}, we see that $f\in \BV(A\cup B)$. 

    From the reasoning above, it follows that $f\in \BV(A_1\cup \dots A_n)$ for every $n$, with
    \[
    |\DIFF f|\mres(A_1\cup \cdots \cup A_n)\le |\DIFF f|\mres A_1\vee\cdots\vee |\DIFF f|\mres A_n.
    \]
Now by (the proof of) item $i)$ of the proof of \cite[Lemma 5.2]{AmbrosioDiMarino14} again, we conclude.
\end{proof}
When the underlying space is a PI space, the theory of BV-functions and sets of finite perimeter enjoy fine properties. We recall here the notion of a PI space and list properties that we are going to use following, see e.g.\ \cite{Ambrosio2001,Ambrosio2002,AmbrosioMirandaPallara04,KinnunenKorteShanmugaligam14} and references therein.

We say that a metric measure space $(\XX,\dist,\mass)$ is \emph{uniformly locally doubling}, provided for any $R>0$ there is ${\sf Doub}(R)>0$ so that
\begin{equation}
   \mass(B_{2r}(x))\le {\sf Doub}(R)\mass(B_r(x)),\qquad \text{for every }x\in\XX,r \in (0,R). \label{eq:doubling}
\end{equation}
In this situation, given a non-negative finite Borel measure $\mu$, the maximal operator
\[
\mathcal{M}_r(\mu)(x) \coloneqq \sup_{s \in (0,r)} \frac{\mu(B_s(x))}{\mass(B_s(x))},\qquad x \in \XX,r \in (0,\infty], 
\]
is well behaved in the following sense:{
for every $R>0,\sigma\ge 1$ there exists a constant $C(R,\sigma)>0$ depending only on the doubling constant, $R$ and $\sigma$ so that the following weak-$L^1$ estimate holds
\begin{equation}
\mass\big( \{x \in B \colon \mathcal M_{2\sigma r} (\mu)(x) >\lambda \} \big) \le C(R,\sigma)\frac{\mu( 3\sigma B)}\lambda,\qquad \forall\lambda>0,
\label{eq:maximal operator BV}
\end{equation}
for every ball $B\subseteq \XX$ with radius $r \le R$}. This is classical and follows by a standard argument using a $5r$-Vitali covering and the doubling assumption on the measure. {If, $R=\infty$, we shall simply denote $\mathcal{M}(\mu)\coloneqq \mathcal{M}_\infty(\mu)$}.

We say that $\XX$ supports a \emph{weak local $(1,1)$-Poincar\'e inequality}, provided there is $\tau_p\ge 1$ and, for every $R>0$, a constant $C(R)>0$ so that, for any $f \in \LIP_{\rm loc}(\XX)$, it holds
\[
\dashint_{B_r(x)}|f-f_{x,r}|\,\dd \mass \le C(R) r\dashint_{B_{\tau_p r}(x)} \lip(f)\,\dd \mass,\qquad \text{for very }x \in \XX,r\in(0,R),
\]
where $f_{x,r}\coloneqq \dashint_{B_r(x)} f\,\dd\mass$. We remark that usually this definition is given with the concept of \emph{upper gradient}, rather than with the notion of local Lipschitz constant on the right-hand side. However, in light of the density result of \cite{AmbrosioGigliSavaredensity}, the two approaches are perfectly equivalent on arbitrary metric measure spaces (this was already known on doubling spaces as a previous result of \cite{Keith03}). The above Poincar\'e inequality upgrades via approximation to functions of bounded variation (see, e.g. \cite[Lemma 1.13]{BonicattoPasqualettoRajala20}): for every $f \in \BV(\XX)$, it holds
\begin{equation}
\int_{B_r(x)}|f-f_{x,r}|\,\dd \mass \le C(R) r|\DIFF f|(B_{\tau_p r}(x)),\qquad  \text{for very }x \in \XX,r\in(0,R). \label{eq:Poincare BV}
\end{equation}
Here we stated the inequality in non-averaged form, for convenience in the rest of the work.

\begin{defn}[PI spaces]
    A metric measure space $(\XX,\dist,\mass)$ is a PI space provided it is locally doubling and supports a weak local $(1,1)$-Poincar\'e inequality.
\end{defn}
We start stating some fine properties of sets of finite perimeter following \cite{Ambrosio2001,Ambrosio2002} on PI spaces. 
Given a Borel subset $E\subseteq \XX$ of a doubling metric measure space and a point $x \in \XX$, let us define the Borel set
\[
\partial^*E \coloneqq \Big\{ x \in \XX \colon \overline{D}(E,x)  >0, \overline{D}(E^c,x) >0\Big\}\subseteq \partial E,
\]
where the upper density and lower density of $E$ at $x$ are respectively defined
as
\[
\overline{D}(E,x)\coloneqq  \limsup_{r\searrow 0} \frac{\mass(E\cap B_r(x))}{\mass(B_r(x))},\qquad  \underline{D}(E,x)\coloneqq \liminf_{r\searrow 0} \frac{\mass({E}\cap B_r(x))}{\mass(B_r(x))} .
\]
When the two above are equal, we denote $D(E,x)$ their common value, called simply the density of $E$ at $x$. The set of $t$ density points of a set $E$, for $t\in[0,1]$ is then defined as
\[
E^{(t)}\coloneqq \{x \in \XX \colon D(E,x)=t\}.
\]
We call $E^{(0)},E^{(1)}$ the \emph{measure theoretic exterior} and \emph{interior}, respectively.
Given a measurable function $f \colon \XX \to \RR$, we define its \emph{approximate lower} and \emph{upper limits} 
\begin{align*}
    & f^\wedge(x) \coloneqq {\rm ap }\liminf_{y\to x}f(y) \coloneqq \sup\big\{ t \in \bar \RR \colon D(\{f<t\},x)=0 \big\},\\
    & f^\vee(x) \coloneqq {\rm ap }\limsup_{y\to x}f(y) \coloneqq \inf\big\{ t \in \bar \RR \colon D(\{f>t\},x)=0 \big\},
\end{align*}
here $\bar \RR = \RR \cup \{\pm \infty\}$ and we adopt the convention $\sup \varnothing = -\infty , \inf \varnothing = +\infty$. We then say that $f$ is approximately continuous at $x$ if $f^\wedge (x) = f^\vee(x) \eqqcolon {\rm ap }\lim_{y\to x}f(y)$. In the case of Borel maps $u \colon \XX \to \YY$ valued on a metric space $(\YY,\rho)$, we say that $u$ is approximately continuous at $x$ if there exists $z \in \YY$ so that $D(\{\rho(u(\cdot),z)>\epsilon\},x)=0$ for all $\epsilon>0$. Denote $\tilde u(x) \coloneqq z$, being $z$ unique if it exists. It can be readily checked that this notion is equivalent to the previous when $\RR=\YY$. Now, set
\[
J_f \coloneqq \{ f^\wedge < f^\vee \},
\]
and define the \emph{precise representative}
\[ \bar f(x) = \frac{f^\wedge(x)+f^\vee(x)}{2},\qquad \forall x \in \XX, \]
adopting the convention that $+\infty-\infty =0$ occuring when $x \notin J_f$. 

We recall the notion of \emph{codimension-one Hausdorff measure}  $\mathcal H^h$. For every $A\subseteq \XX$ and $\delta>0$, set
\[
    \mathcal H_\delta ^h  (A)\coloneqq \inf\Big\{\sum_{i\in\NN} \frac{\mass(B_{r_i}(x_i))}{r_i} \colon A\subseteq \bigcup_{i\in\NN}B_{r_i}(x_i), r_i\le \delta \Big\}.
\]
Then, define
\[
   \mathcal H^h(A)\coloneqq \lim_{\delta \searrow 0}\mathcal H^h_\delta(A),\qquad \forall A\subseteq \XX,
\]
which, on PI spaces, is a Borel regular outer measure. From \cite{Ambrosio2001,Ambrosio2002} and \cite{AmbrosioMirandaPallara04}, we know the validity of the following result. This has been established under a global doubling assumption but actually holds under the following hypothesis.

\begin{thm}[Representation of the perimeter]\label{thm:repr_per}
Let \((\XX,\dist,\mass)\) be a PI space and let \(\varnothing\neq\Omega\subseteq\XX\) be an open set. Let \(E\subseteq\Omega\) be a Borel set having finite perimeter in \(\Omega\).
Then \({\rm Per}(E,\cdot)\) is concentrated on \(\partial^*E\cap\Omega\) and it holds that
\(\HH^h(\partial^*E\cap\Omega)<+\infty\). Moreover, there are constants
\(\gamma,C>0\) and a Borel function \(\theta_{E}\colon\Omega\to[\gamma,C]\) such that
\[
{\rm Per}(E,\cdot)=\theta_{E}\HH^h\mres{\partial^*E\cap\Omega}.
\]
Here, $C$ and $\gamma$ depend on the PI-parameters.
\end{thm}

\begin{rem}\label{rem:S max}{\rm
From the above and \cite[Theorem 5.4]{Ambrosio2002}, we point out the following important fact. There exists a natural number \(n_0\in\NN\), depending only on the PI-parameters so that \(\HH^h(\bigcap_{i=1,...,S} E_i)=0\) whenever $S\ge n_0$ and $(E_i)\subseteq \Omega$ are disjoint sets of finite perimeter.
\fr}\end{rem}

For $f \in \BV(\XX)$ defined on PI metric measure space $\XX$ we have from \cite{KinnunenKorteShanmugaligam14} that 
\[
-\infty < f^\wedge \le f^\vee <+\infty,\qquad \HH^h\text{-a.e.}
\]
In particular, defining $\XX_f\coloneqq \{-\infty < f^\wedge \le f^\vee <+\infty\}$, we have $|\DIFF f|(\XX\setminus \XX_f)=0$. Moreover,  we recall from \cite{AmbrosioMirandaPallara04}
that \(J_f\) is a countable union of essential boundaries of sets of finite perimeter, as well as the following identity:
\begin{equation}
  |\DIFF f|\mres{J_f} = (f^\vee-f^\wedge)\theta_f \HH^h\mres{J_f},\label{eq:Df on Sf}  
\end{equation}
where $\theta_f\colon J_f\to[\gamma,C]$ is the Borel function defined as $ \theta_f(x)= \dashint_{f^\wedge(x)}^{f^\vee(x)}\theta_{\{f>s\}}(x)\,\dd s$.
Notice that if \(E\subseteq\XX\) is of finite perimeter and \(\mass(E)<+\infty\), then \(J_{\chi_E}=\partial^*E\) and \(\theta_{\chi_E}=\theta_E\) on \(\partial^*E\).
We shall use fine results such as \eqref{eq:Df on Sf} also for functions of locally bounded variations defined on open sets. Given the local nature of these properties, their verification is the same.
\subsection{Isotropic PI spaces}
Following \cite{AmbrosioMirandaPallara04}, we say that a PI space \((\XX,\dist,\mass)\) is \emph{isotropic} if
\begin{equation}\label{eq:isotropic}
\theta_E=\theta_F\quad\text{ holds }\mathcal H^h\text{-a.e.\ on }\partial^*E\cap\partial^*F
\end{equation}
whenever \(E,F\subseteq\XX\) are sets of finite perimeter satisfying \(E\subseteq F\).
\begin{lem}\label{lem:glue_theta}
Let \((\XX,\dist,\mass)\) be an isotropic PI space. Let \((E_i)_i\) be a sequence of sets of finite perimeter in \(\XX\) and define \(\Gamma\coloneqq\bigcup_{i\in\NN}\partial^*E_i\).
Then there exists a Borel function \(\theta\colon\Gamma\to[\gamma,C]\) that is unique up to \(\HH^h\)-a.e.\ equality and satisfies the following property: given any \(f\in\BV(\XX)\),
\[
\theta_f=\theta\quad\text{ holds }\HH^h\text{-a.e.\ on }J_f\cap\Gamma.
\]
\end{lem}
\begin{proof}
We subdivide the proof into three steps.
\medskip\\\textbf{Step 1.} We claim that if \(E,F\subseteq\XX\) are given sets of finite perimeter, then we have that
\begin{equation}\label{eq:improved_isotr_aux1}
\theta_E=\theta_F\quad\text{ holds }\HH^h\text{-a.e.\ on }\partial^*E\cap\partial^*F.
\end{equation}
Notice that \eqref{eq:improved_isotr_aux1} is -- a priori -- not guaranteed by the isotropicity assumption, since we are not requiring \(E\subseteq F\) nor \(F\subseteq E\).
By trivial density estimates, for any two sets $A,B\subseteq\XX$ of finite perimeter we get
\begin{equation}\label{sdcvascs}
    \partial^*(A\cap B)\cup\partial^* (A\cup B)\subseteq\partial^*A\cup\partial^* B
\end{equation}
and, being the space isotropic,
\[\begin{split}
\theta_{A\cap B}=\theta_A&\quad\text{ holds }\HH^h\text{-a.e.\ on }\partial^*A\cap\partial^*(A\cap B),\\
\theta_{A\cap B}=\theta_B&\quad\text{ holds }\HH^h\text{-a.e.\ on }\partial^*B\cap\partial^*(A\cap B),
\end{split}\]
so that 
\begin{equation}\label{eq:isotr_equiv_aux2}
\theta_A=\theta_B\quad\text{ holds }\HH^h\text{-a.e.\ on }\partial^*A\cap\partial^*B\cap\partial^*(A\cap B).
\end{equation}
Plugging \((A,B)=(E\cap F,E^c\cap F)\) and \((A,B)=(E\cap F,E\cap F^c)\) in \eqref{sdcvascs}, we get
\begin{equation}\label{eq:isotr_equiv_aux1}
\partial^*E\cap\partial^*F\subseteq\partial^*(E\cap F)\cup\partial^*(E^c\cap F)\cup\partial^*(E\cap F^c).
\end{equation}
Plugging \((A,B)=(E,F)\), \((A,B)=(E^c,F)\) and \((A,B)=(E,F^c)\) in \eqref{eq:isotr_equiv_aux2}, we obtain
\begin{equation}\label{eq:isotr_equiv_aux3}\begin{split}
\theta_E=\theta_F&\quad\text{ holds }\HH^h\text{-a.e.\ on }\partial^*E\cap\partial^*F\cap\partial^*(E\cap F),\\
\theta_{E^c}=\theta_F&\quad\text{ holds }\HH^h\text{-a.e.\ on }\partial^*({E^c})\cap\partial^*F\cap\partial^*(E^c\cap F),\\
\theta_E=\theta_{F^c}&\quad\text{ holds }\HH^h\text{-a.e.\ on }\partial^*E\cap\partial^*({F^c})\cap\partial^*(E\cap F^c),
\end{split}\end{equation}
respectively. Since \(\theta_A\) and \(\partial^*A\) are invariant under passing to the complement of \(A\), we deduce from \eqref{eq:isotr_equiv_aux1}
and \eqref{eq:isotr_equiv_aux3} that \(\theta_E=\theta_F\) holds \(\HH^h\)-a.e.\ on \(\partial^*E\cap\partial^*F\), thus proving \eqref{eq:improved_isotr_aux1}.
\medskip\\\textbf{Step 2.} Fix any \(f\in\BV(\XX)\). Fix a dense sequence \((\lambda_n)\) in \(\RR\) such that each \(\{\bar f>\lambda_n\}\) is of finite perimeter
and a Borel partition \((G_n)\) of \(J_f\) with \(G_n\subseteq\{f^\wedge<\lambda_n<f^\vee\}\). We claim that
\begin{equation}\label{eq:improved_isotr_aux2}
\theta_f=\sum_{n\in\NN}\nchi_{G_n}\theta_{\{\bar f>\lambda_n\}}\quad\text{ holds }\HH^h\text{-a.e.\ on }J_f.
\end{equation}
For any \(n\in\NN\), define \(\Gamma_n\coloneqq\big\{(x,t)\in\XX\times\RR\colon x\in G_n,\,t\in(f^\wedge(x),f^\vee(x))\big\}\). Notice that
\begin{equation}\label{eq:isotr_fc_aux}
\Gamma_n=\bigcup_{\substack{q,\tilde q\in\QQ:\\q<\lambda_n<\tilde q}}G_n^{q,\tilde q}\times(q,\tilde q),
\quad\text{ where we set }G_n^{q,\tilde q}\coloneqq\big\{x\in G_n\colon f^\wedge(x)<q,\,f^\vee(x)>\tilde q\big\}.
\end{equation}
If \(t\in(q,\tilde q)\) is given, then \(G_n^{q,\tilde q}\subseteq\partial^*\{\bar f>t\}\), so that
\(\theta_{\{\bar f>t\}}=\theta_{\{\bar f>\lambda_n\}}\) holds \(\HH^h\)-a.e.\ on \(G_n^{q,\tilde q}\)
by the isotropicity assumption. Thanks to \eqref{eq:isotr_fc_aux} and to Fubini's Theorem, we deduce that
for \(\HH^n\)-a.e.\ \(x\in G_n\) it holds that \(\theta_{\{\bar f>t\}}(x)=\theta_{\{\bar f>\lambda_n\}}(x)\)
for \(\mathcal L^1\)-a.e.\ \(t\in(f^\wedge(x),f^\vee(x))\). Therefore,
\[
\theta_f(x)=\dashint_{f^\wedge(x)}^{f^\vee(x)}\theta_{\{\bar f>t\}}(x)\,\dd t=\theta_{\{\bar f>\lambda_n\}}(x)
\quad\text{ for }\HH^h\text{-a.e.\ }x\in G_n,
\]
whence \eqref{eq:improved_isotr_aux2} follows.
\medskip\\\textbf{Step 3.} Thanks to \textbf{Step 1}, we can find a Borel function \(\theta\colon\Gamma\to[\gamma,C]\) such that
\begin{equation}\label{eq:constr_theta}
\theta_{E_i}=\theta\quad\text{ holds }\HH^h\text{-a.e.\ on }\partial^*E_i,\text{ for every }i\in\NN.
\end{equation}
Finally, we claim that if \(f\in\BV(\XX)\) is given, then for every \(i\in\NN\) we have that
\begin{equation}\label{eq:improved_isotr_aux3}
\theta_f=\theta_{E_i}\quad\text{ holds }\HH^h\text{-a.e.\ on }J_f\cap\partial^*E_i,
\end{equation}
which (recalling \eqref{eq:constr_theta}) yields the statement. To prove the claim, fix a dense sequence \((\lambda_n)\) in \(\RR\) such that \(\{\bar f>\lambda_n\}\)
is of finite perimeter for every \(n\in\NN\). Take a Borel partition \((G_n)\) of \(J_f\) such that \(G_n\subseteq\{f^\wedge<\lambda_n<f^\vee\}\)
for every \(n\in\NN\). Then for any \(n\in\NN\) we have that
\[
\theta_{\{\bar f>\lambda_n\}}=\theta_{E_i}\quad\text{ holds }\HH^h\text{-a.e.\ on }G_n\cap\partial^*E_i\subseteq\partial^*\{\bar f>\lambda_n\}\cap\partial^*E_i
\]
thanks to \textbf{Step 1}. Using \textbf{Step 2}, we thus conclude that \(\HH^h\)-a.e.\ on \(J_f\cap\partial^*E_i\) we have
\[
\theta_f=\sum_{n\in\NN}\nchi_{G_n\cap\partial^*E_i}\theta_{\{\bar f>\lambda_n\}}=\sum_{n\in\NN}\nchi_{G_n\cap\partial^*E_i}\theta_{E_i}=\theta_{E_i},
\]
which proves the validity of \eqref{eq:improved_isotr_aux3}.
\end{proof}
\begin{cor}\label{cor:isotrop_BV}
Let \((\XX,\dist,\mass)\) be an isotropic PI space. Let \(f,g\in\BV(\XX)\) be given. Then
\begin{equation}\label{eq:isotrop_BV_cl1}
\theta_f=\theta_g\quad\text{ holds }\HH^h\text{-a.e.\ on }J_f\cap J_g.
\end{equation}
In particular, if \(E,F\subseteq\XX\) are given sets of finite perimeter, then
\begin{equation}\label{eq:isotrop_BV_cl2}
\theta_E=\theta_F\quad\text{ holds }\HH^h\text{-a.e.\ on }\partial^*E\cap\partial^*F.
\end{equation}
\end{cor}
\begin{proof}
Recall that \(J_f\cup J_g\) is a countable union of essential boundaries. Hence, we know from Lemma \ref{lem:glue_theta} that there exists a Borel function
\(\theta\colon J_f\cup J_g\to[\gamma,C]\) such that
\[
\theta_f=\theta\quad\text{ holds }\HH^h\text{-a.e.\ on }J_f,\qquad\theta_g=\theta\quad\text{ holds }\HH^h\text{-a.e.\ on }J_g,
\]
whence \eqref{eq:isotrop_BV_cl1} follows. Finally, \eqref{eq:isotrop_BV_cl2} follows either from (a localized version of) \eqref{eq:isotrop_BV_cl1} or from \textbf{Step 1} of the proof of Lemma \ref{lem:glue_theta}.
\end{proof}
Notice that, thanks to Theorem \ref{thm:repr_per}, it can be checked that if $\XX$ is an isotropic PI space, $\varnothing \neq \Omega\subseteq \XX$ is open and $f,g\in\BV(\Omega)$ are given, then \eqref{eq:isotrop_BV_cl1} localizes on $\Omega$:
\begin{equation}
    \theta_f = \theta_g\quad\text{ holds }\HH^h\mres\Omega\text{-a.e.\ on }J_f\cap J_g.
\label{eq:isotropicOmega}
\end{equation}
Finally, following \cite{BonicattoPasqualettoRajala20}, we say that \((\XX,\dist,\mass)\) has the \emph{two-sidedness property} provided 
\[
\mathcal H^h(\partial^*E\cap\partial^*F\cap\partial^*(E\cup F))=0\]
whenever \(E,F\subseteq\XX\) are disjoint sets of finite perimeter. For a given PI space \((\XX,\dist,\mass)\), the following implications hold:
\[
 \substack{\mathcal H^h(\partial^*E\setminus E^{(1/2)})=0\text{ for} \\ \text{all }E\text{ of finite perimeter}}
\quad\Longrightarrow\quad\text{two-sidedness property}\quad\Longrightarrow\quad\text{isotropic.}
\]
The converse implications are, in general, false.
\section{Three notions of \texorpdfstring{$\BV$}{BV}-maps}
In this part, we give three definitions of maps of bounded variation, discuss their main properties and prove our main result Theorem \ref{thm:main result 1} as a consequence of Theorem \ref{thm:V is a measure} below.

\subsection{Post-composition with Lipschitz functions}
\begin{defn}[BV-maps via post-composition]\label{def:Du Resh}
Let $(\XX,\dist,\mass)$ be a metric measure space, let $\varnothing \neq \Omega\subseteq \XX$ be open and let $(\YY,\rho,\bar y)$ be a pointed separable metric space. Let $\mathcal{F}\subseteq \LIP(\YY)$ be a family of $1$-Lipschitz functions.

We say that $u \in L^1_{\rm loc}(\Omega,\YY)$ belongs to the space $\BV_{\mathcal{F},{\rm loc}}(\Omega,\YY)$ provided that there exists a non-negative Borel measure $\mu$ in $\Omega$ such that 
\begin{equation}
    |\DIFF (\varphi \circ u)|\le \mu,\qquad\text{for every $\varphi \in \mathcal{F}$},\label{eq:minimal mu}
\end{equation}
(in particular, $\varphi \circ u\in \BV_{\rm loc}(\Omega)$ for every $\varphi\in\mathcal{F}$).
The minimal measure $\mu$ satisfying \eqref{eq:minimal mu} is denoted by $|\DIFF u|_\mathcal{F}$ and called the total variation of $u$. 

If $u \in L^1(\Omega,\YY_{\bar y})$ and any $\varphi \in \mathcal{F}$ is so that $\varphi(\bar y)=0$, then we write $u \in \BV_\mathcal{F}(\Omega,\YY_{\bar y})$ provided $|\DIFF u|_\mathcal{F}(\Omega)<\infty$.

When $\mathcal{F} = \{ \varphi \colon {\rm Lip}(\varphi) \le 1\}$ (resp. $\mathcal{F} = \{ \varphi - \varphi(\bar y)  \colon {\rm Lip}(\varphi) \le 1\}$), we simply write $\BV_{\rm loc}(\Omega,\YY)$ (resp. $\BV(\Omega,\YY_{\bar y})$) and $|\DIFF u|$.

\end{defn}

Elementary yet effective examples of families $\mathcal{F}$ can be built from the distance function as follows: for a countable subset $(y_n)_n\subseteq\YY$
$$
\mathcal F\defeq \left\{\rho(\,\cdot\,,y_n)\right\}_{n\in\NN}
\quad \text{or}\quad
\mathcal F\defeq \left\{\rho(\,\cdot\,,y_n)\wedge m\right\}_{n,m\in\NN}.
$$
In the following lemma, we show that a countable family of special Lipschitz functions in the target metric space always exists capable of detecting the membership and the total variation of a map. Assuming an isotropic PI source, we also single out a property for this  to happen.
\begin{prop}\label{lem:family F isotropic}
Let $(\XX,\dist,\mass)$ be a metric measure space, let $\varnothing \neq \Omega\subseteq \XX$ be open and let $(\YY,\rho)$ be separable metric space. Then, there is a countable family $\mathcal{F}\subseteq \LIP(\YY)$ of $1$-Lipschitz function so that: $u \in \BV_{\mathcal{F},{\rm loc}}(\Omega,\YY)$ if and only if $u \in \BV_{\rm loc}(\Omega,\YY)$ and we have
$$|\DIFF u|_{\mathcal{F}}=|\DIFF u|.$$

Moreover, if $\XX$ is an isotropic PI space and $\mathcal{F}\subseteq \LIP(\YY)$ is \emph{any} family of $1$-Lipschitz  functions satisfying
\begin{equation}\label{eq:property F}
\rho(y,z) = \sup_{\varphi \in \mathcal{F}}  |\varphi(y)-\varphi(z)|\qquad \text{for every } y,z \in \YY,
\end{equation}
then the same conclusion holds.
\end{prop}
\begin{proof}
We subdivide the proof into different steps and, before starting, let us fix some notation. In the sequel, we consider fixed $(y_n)_n\subseteq \YY$ countable and dense. Moreover, given $\mathcal{H}\subseteq\LIP(\YY)$ \emph{any} family of $1$-Lipschitz functions, we always indicate in this proof
$$
\mathcal H'\defeq\{a\phi+b:\phi\in\mathcal{H},a\in \{0,\pm1\},b\in\RR\},\quad \mathcal{H}''\defeq\{\phi_1\vee\dots\vee\phi_n:n\in\NN, \phi_1,\dots\phi_n\in\mathcal{H}'\}.$$
We then fix the families
\[
\mathcal{G}\coloneqq \{\rho(\cdot,y)\colon y \in\YY \},\qquad \mathcal{G}_\QQ \coloneqq \{\rho(\cdot,y_n)\colon n \in\NN \},
\]
and, with a little abuse of notation, we shall indicate accordingly
\[
  \mathcal{G}_\QQ' \coloneqq \{a\varphi+b \colon \varphi\in\mathcal{G}_\mathbb{Q}, a \in \{0,\pm1\},b\in\QQ \},\quad 
   \mathcal{G}_\QQ'' \coloneqq \{\phi_1\vee\dots\vee\phi_n:n\in\NN, \phi_1,\dots\phi_n\in\mathcal{G}_\QQ'\}.
\]
\medskip\\\textbf{Step 1: proof for arbitrary $\XX$}.
Let us consider first the case of an arbitrary metric measure space $\XX$. We need to find a countable family $\mathcal{F}$ satisfying
\begin{equation}
   |\DIFF u|=|\DIFF u|_{\mathcal F}, \label{eq:claim F arbitrary}
\end{equation} 
(here and after, we are tacitly proving also the BV membership when proving the equalities). We claim that
\[
\mathcal{F}\coloneqq \mathcal{G}''_\QQ
\]
does the job. Clearly, it is a countable {family of} $1$-Lipschitz functions by construction. We are left with the verification of \eqref{eq:claim F arbitrary}. First notice that a standard approximation procedure in combination with the lower-semicontinuity of the total variation gives $|\DIFF u|_{\mathcal{G}''} = |\DIFF u|_{\mathcal{G}''_\QQ}$. Thus, our claim \eqref{eq:claim F arbitrary} will be proved if we show
\begin{equation}
   |\DIFF u|=|\DIFF u|_{\mathcal{G}''}. \label{eq:claim 1 F}
\end{equation} 
As $|\DIFF u|_{\mathcal{G}''} \le |\DIFF u|$ always holds, we prove the reverse implication, i.e.\ for any $\varphi:\YY\rightarrow\RR$ $1$-Lipschitz function it holds
$$
|\DIFF(\varphi\circ u)|\le |\DIFF u|_{\mathcal{G}''}.
$$
By lower semicontinuity of the total variation measure, it is not restrictive to suppose $\varphi$ to be bounded. 
 Define then $\varphi_0\defeq -\Vert\varphi\Vert_\infty$ and 
$$\varphi_n'\defeq (\varphi(y_n)-\rho(\,\cdot\,,y_n))\qquad\text{then}\qquad\varphi_n\defeq\bigvee_{0\le k\le n}\varphi_n'.$$ Notice that $(\varphi_n')_n\subseteq\mathcal{G}'$ so that $(\varphi_n)_n\subseteq\mathcal{G}''$. It is easy to show that $\varphi_n\nearrow\varphi$ pointwise, then, by $\Vert\varphi_n\Vert_\infty\le \Vert\varphi\Vert_\infty$ we have $\varphi_n \circ u\rightarrow\varphi\circ u$ in $L^1_{\rm loc}(\XX)$ and the claim \eqref{eq:claim 1 F} is proved by the lower semicontinuity of the total variation. By what was said before, \eqref{eq:claim F arbitrary} follows concluding the proof for arbitrary $\XX$.
\medskip\\\textbf{Step 2: extension of families for $\XX$ isotropic PI}.
From here, we assume $ \XX$ to be isotropic PI. In this step, we prove that for any family $\mathcal{H}$ of $1$-Lipschitz functions, we have 
$$
|\DIFF u|_{\mathcal{H''}}=|\DIFF u|_{\mathcal{H}}.
$$
As trivially $|\DIFF u|_{\mathcal{H''}}\ge|\DIFF u|_{\mathcal{H'}}=|\DIFF u|_{\mathcal{H}}$, it is enough to show that for every $f,g\in\BV_{\rm loc}(\Omega)$, it holds that
$$
|\DIFF ( f \vee g)|\le |\DIFF f|\vee|\DIFF g|,
$$
(notice that $f\vee g\in\ \BV_{\rm loc}(\Omega)$ with  $|\DIFF (f\vee g)|\le |\DIFF f|+|\DIFF g|$). To this aim, notice first that we can assume with no loss of generality that $f,g\in L^\infty(\Omega)$. Notice also that $$f^\wedge \vee g^\wedge\le (f\vee g)^\wedge\le (f\vee g)^\vee = f^\vee\vee g^\vee,$$ then
$(f\vee g)^\vee-(f\vee g)^\wedge\le (f^\vee-f^\wedge)\vee (g^\vee-g^\wedge)$. In particular, \(J_{f\vee g}\subseteq J_f\cup J_g\). Moreover, if $x\notin J_f\cup J_g$, then $\overline{f\vee g}(x)=\bar f(x)\vee \bar g(x)$.
Now, using \eqref{eq:Df on Sf} and Corollary \ref{cor:isotrop_BV} we obtain that
\begin{align*}
	|\DIFF(f\vee g)|\mres J_{f\vee g}&=\big((f\vee g)^\vee-(f\vee g)^\wedge\big)\theta_{f\vee g}\HH^h\mres J_{f\vee g}\\
 &\leq\big((f^\vee-f^\wedge)\vee(g^\vee-g^\wedge)\big)\theta_{f\vee g}\HH^h\mres\big((J_{f\vee g}\cap J_f)\cup(J_{f\vee g}\cap J_g)\big)\\
  &=\big((f^\vee-f^\wedge)\theta_{f\vee g}\HH^h\mres(J_{f\vee g}\cap J_f)\big)\vee\big((g^\vee-g^\wedge)\theta_{f\vee g}\HH^h\mres(J_{f\vee g}\cap J_g)\big)\\
 &=\big((f^\vee-f^\wedge)\theta_f\HH^h\mres(J_{f\vee g}\cap J_f)\big)\vee\big((g^\vee-g^\wedge)\theta_g\HH^h\mres(J_{f\vee g}\cap J_g)\big)\\
 &\leq\big((f^\vee-f^\wedge)\theta_f\HH^h\mres J_f\big)\vee\big((g^\vee-g^\wedge)\theta_g\HH^h\mres J_g\big)\\
 &=(|\DIFF f|\mres J_f)\vee(|\DIFF g|\mres J_g).
\end{align*}
Also, using repeatedly \cite[Lemma 3.4]{KinnunenKorteShanmugaligam14} together with Chebyshev inequality 
we have that for $\HH^h$-a.e.\ $x\notin J_f\cup J_g$ such that $\bar f\ge \bar g$, then $x\notin J_{(f\vee g)-f}$ and $\overline{(f\vee g)-f}(x)=0$. Then, by the coarea formula, if we set $B_f\defeq \{\bar f\ge \bar g\}\setminus(J_f\cup J_g)$,
\begin{align*}
	|\DIFF ((f\vee g)-f)|(B_f)=\int_\RR \per(\{(f\vee g)-f)>t\},B_f)\dd t=0,
\end{align*}
as for $\HH^h$-a.e.\ $x\in B_f\cap\partial^*\{(f\vee g)-f)>t\}$, then $t=\overline{(f\vee g)-u}(x)=0$. Therefore, we have obtained 
$|\DIFF (f\vee g)|\mres B_f=|\DIFF f|\mres B_f$. By symmetry, $|\DIFF (f\vee g)|\mres B_g=|\DIFF g|\mres B_g$, where $B_g\defeq \{\bar g\ge \bar f\}\setminus(J_f\cup J_g)$. All in all, we have proved the claim, taking into account that $|\DIFF (f\vee g)|\mres (\XX\setminus J_{f\vee g})(J_f\cup J_g)=0$.
\medskip\\\textbf{Conclusion}. Let $\XX$ be isotropic and let $\mathcal{F}$ be any family of $1$-Lipschitz functions satisfying \eqref{eq:property F}. We need to show that
\[
|\DIFF u| = |\DIFF u|_{\mathcal{F}}.
\]
From \eqref{eq:claim 1 F} in \textbf{Step 1} and thanks to \textbf{Step 2}, we directly achieve $|\DIFF u| = |\DIFF u|_{\mathcal{G}}$ and $|\DIFF u|_{\mathcal{F}} = |\DIFF u|_{\mathcal{F}''}$. So, the proof will be completed if we show
\[
|\DIFF u|_{\mathcal{G}} = |\DIFF u|_{\mathcal{F}''}.
\]
As $|\DIFF u|_{\mathcal{F''}} \le |\DIFF u|= |\DIFF u|_{\mathcal{G}}$ always holds, it is then enough to show that for any $\bar  y\in\YY$, then
$$|\DIFF\rho(u(\,\cdot\,), \bar y)|\le |\DIFF u|_{\mathcal{F}''}.$$
For every $(n,k)\in\NN^2$, take $\varphi_{n,k}\in\mathcal{F}''$ such that 
$$
\varphi_{n,k}(y_n)-\varphi_{n,k}( \bar y)\ge \rho(y_n,\bar y)-k^{-1},
$$
this possibility is ensured by the assumption \eqref{eq:property F} on $\mathcal{F}$.
Now fix $l:\NN\rightarrow\NN^2$, an enumeration of $\NN^2$, and define 
$$\varphi_m\defeq\bigvee_{a\le m}(\varphi_{l(a)}-\varphi_{l(a)}( \bar y))\vee 0.$$
Notice that $(\varphi_m)_m\subseteq\mathcal{F}''$ and $\varphi_m \nearrow\rho(\,\cdot\,,\bar y)$. Then we conclude as at the end of \textbf{Step 1}.
\end{proof}
We show now a Leibniz-type result proving that any $\BV_{\rm loc}$-map { restricted to any compact set can be extended as a global BV-map}.
\begin{prop}[Cut-off]\label{prop:cut off}
Let $(\XX,\dist,\mass)$ be a metric measure space, let $\varnothing \neq \Omega\subseteq \XX$ be open and let $(\YY,\rho,\bar y)$ be a pointed separable metric space. Fix $u \in \BV_{\rm loc}(\Omega,\YY)$. Then, for every compact $\varnothing\neq K\subseteq \Omega$, there is $0< r_0\coloneqq r_0(u,K) \le \dist(K,\XX\setminus \Omega)$ (set to $+\infty$ if $\Omega=\XX$) so that, setting
\[
w_r \coloneqq \begin{cases} u\qquad&\text{on }\Omega_r \coloneqq \{ \dist(\,\cdot\,,K)<r \}, \\  \bar y\qquad&\text{otherwise},\end{cases}
\]
it holds
\begin{equation}\label{csdcads}
|\DIFF w_r| \le \rho(u,\bar y) |\DIFF \nchi_{\Omega_r}| + |\DIFF u|\mres{\Omega_r} \qquad \text{for a.e.\ }r\in (0,r_0),
\end{equation}
as measures on $\XX$,  where we implicitly state that for $r$ satisfying \eqref{csdcads}, $w_r\in \BV(\XX,\YY_{\bar y})$ and we notice that \eqref{csdcads} makes sense thanks to coarea.

In particular, for any $x\in \Omega$, there exists $w_x\in \BV(\XX,\YY_{\bar y})$ with $u=w_x$ $\mass$-a.e.\  on a neighbourhood of $x$.
\end{prop}
\begin{proof}
    {Let us consider a fixed Borel representative of $u$}. Let us also take $\mathcal{F}\subseteq \LIP(\YY)$ the countable family of $1$-Lipschitz functions given by Lemma \ref{lem:family F isotropic} and consider an enumaration $(\varphi_i)_{i \in \NN}$ of the family of functions $\{ \varphi - \varphi(\bar y) \colon \varphi \in \mathcal{F}\}$.

    By definition of maps of locally bounded variation and since $K$ is compact, there is an open neighborhood $ U\subseteq \XX$ so that $K \Subset U$ and, taking into account Lemma \ref{footnote9}, $u \in \BV(U,\YY_{\bar y})$. Set $\Omega_r \coloneqq \{\dist(\cdot, K)<r\}$ and, for every $i \in\NN$, consider $ \varphi_i \circ u \in \BV(U)$ and optimal approximations $(f^i_n)\subseteq \LIP_{\rm loc}(U)$ so that
    \[
        f^i_n \to \varphi_i \circ u \text{ in }L^1(U),\qquad \lip(f^i_n)\mass \to |\DIFF (\varphi_i\circ u)|,
    \]
    in duality with $C_b(U)$, as $n$ goes to infinity. Without loss of generality, we can suppose $f_n^i$ to be bounded functions (up to a truncation and diagonalization argument). We then have that a.e.\ $r \in (0,\dist(\XX\setminus U,K)$ is so that
    \[
    |\DIFF u|(\partial \Omega_r)=0,\qquad {\rm Per}(\Omega_r)<\infty, \qquad  f^i_n \to \varphi_i\circ u \text{ in }L^1(|\DIFF \nchi_{\Omega_r}|), 
    \]
    as $n$ goes to infinity, for all $i\in\NN$. The first two claims are standard, the latter is due to coarea formula.

    Let $r_0\coloneqq r_0(u,K) \coloneqq \dist(K,\XX\setminus U) \le \dist(K,\XX\setminus \Omega)$ and, for any other $r \in (0,r_0)$ satisfying the above (a.e.\ $r$ is suitable),  consider $\eta_k \subseteq \LIP_{\rm loc}(\XX)$ so that
    \[
        \eta_k \to \nchi_{\Omega_r} \text{ in }L^1(\XX),\qquad \lip(\eta_k)\mass \to |\DIFF \nchi_{\Omega_r}|,
    \]
    in duality with $C_b(U)$ as $k$ goes to infinity. Without loss of generality, we can suppose $\eta_k$ to be nonnegative and bounded functions (up to replacing $\eta_k$ with $1\vee (\eta_k \wedge 0)$). By the Leibniz rule for the local lipschitz constant, we therefore get for every $i,k,n\in\NN$
    \[
    \lip(\eta_kf_n^i) \le \eta_k \lip(f_n^i)\mass + |f_n^i| \lip(\eta_k)\mass.
    \]
    By lower semicontinuity of the total variation, sending $k$ to infinty gives
    \[
    |\DIFF(\nchi_{\Omega_r} f_n^i)| \le \nchi_{\Omega_r}\lip(f_n^i)\mass + |f_n^i||\DIFF \nchi_{\Omega_r}|,
    \]
    as measures on $\XX$. Now, by our choice of $r$, we have for all $i \in\NN$
    \[
    \nchi_{\Omega_r}\lip(f_n^i)\mass \to |\DIFF (\varphi_i \circ u) |\mres{\Omega_r},\qquad |f_n^i| |\DIFF \nchi_{\Omega_r}| \to |\varphi_i|\circ u \, |\DIFF \nchi_{\Omega_r}| ,
    \]
    in duality with $C_b(\XX)$, as $n$ goes to infinity, using in the first that $|\DIFF (\varphi_i\circ u)|(\partial B_r)=0$. Thus, again by lower semicontinuity, we reach
    \[
    |\DIFF( \nchi_{\Omega_r} (\varphi_i\circ u))| \le |\DIFF(\varphi_i \circ u)|\mres{\Omega_r} +  |\varphi_i|\circ u \, |\DIFF \nchi_{\Omega_r}| 
    \]
    as measures on $\XX$. By our choice of $\varphi_i$, and since $\nchi_{\Omega_r}(\varphi_i \circ u) = \varphi_i \circ w_r$ $\mass$-a.e., we deduce
    \[
    |\DIFF( \varphi_i \circ w_r)| \le |\DIFF u|\mres{\Omega_r} + \rho(u,\bar y)|\DIFF \nchi_{\Omega_r}| ,
    \]
    having also used that $\varphi_i$ are $1$-Lipschitz. By arbitrariness of $i$, the above gives automatically that $w_r \in \BV(\XX,\YY_{\bar y})$ and the conclusion.    
\end{proof}

\subsection{Relaxation with simple maps}
Here we revisit the notion of BV-maps of Ambrosio in \cite{ambmetric} by means of relaxation with locally simple maps. We start by defining the notion of simple maps.
\begin{defn}[Simple maps]
    Let $(\XX,\dist,\mass)$ be a metric measure space, let $\varnothing \neq \Omega \subseteq \XX$ be open and let $(\YY,\rho)$ be a separable metric space. A Borel map $u \colon \Omega \to  \YY$ is \emph{simple} provided it has a finite range, i.e.\ there exists a finite integer $n\in \NN$, a Borel disjoint partition $(E_i)_{i=1}^n$ of $\Omega$, and points $(y_i)_{i=1}^n \subseteq \YY$ such that $u(x) = y_i$ for $x \in E_i$ and for every $i=1,...,n$.

    Moreover, $u$ is called \emph{locally simple}, provided for every point $x$ there is a neighbourhood $U\ni x $ so that $u$ is simple on $U$.
\end{defn}
We denote by $\SS(\Omega,\YY),\SS_{\rm loc}(\Omega,\YY)$ the collection of simple and locally simple maps, respectively. Given $u \in \SS(\Omega,\YY)$, we will often write its expression with the formal sum
\[
u = \sum_{i=1}^n y_i\nchi_{E_i}.
\]
When $\YY=\RR$, as usual, we drop the dependence on $\YY$ in the notation and refer to any $u \in \SS(\Omega)$ as a simple function. Notice that when $\XX$ is proper we have
    \[
    u \in \SS_{\rm loc}(\Omega,\YY) \qquad \Leftrightarrow \qquad u \in \SS(U,\YY) \qquad \forall U\Subset \Omega.
    \]
\begin{rem}\label{rem:simple map bv}
    It always holds that $\SS_{\rm loc}(\Omega,\YY)\subseteq L^1_{\rm loc}(\Omega,\YY)$. In particular, it makes sense to speak of the total variation $|\DIFF u|$ of a (locally) simple map. We have the following simple characterization: given $u \in \SS(\Omega,\YY)$, say $u=\sum_{i=1}^n y_i \nchi_{E_i}$, it holds
    \[
    |\DIFF u|(\Omega)<\infty \quad \text{if and only if} \quad {\rm Per}(E_i)<\infty,\quad \text{for all }i=1,...,n,
    \]
    in the case in which the $y_i$'s are all distinct.
    Assume that $|\DIFF u|(\Omega)<\infty$, fix any $i \in \{1,..,n\}$ and take a $1$-Lipschitz non-negative function $\varphi \colon \YY\to \RR$ so that $\varphi(y_i)\neq 0$ and $\varphi(y_j)=0$ for all $i\neq j$. By coarea, we have
    \[
    \varphi(y_i){\rm Per}(E_i) = \int {\rm Per }(\{ \varphi\circ u >t\})\,\dd t = |\DIFF (\varphi\circ u)|(\Omega) \le |\DIFF u|(\Omega)<\infty. \]
    This proves the first implication. For the converse, simply notice for any $1$-Lipschitz function $\varphi \colon \YY\to \RR$ that $|\DIFF (\varphi \circ u)| \le \sum_{i=1}^n |\varphi(y_i)|{\rm Per}(E_i,\cdot)$.\fr
\end{rem}  

We are ready to give the notion of BV-map via relaxation.
\begin{defn}\label{def:Vu}
Let $(\XX,\dist,\mass)$ be a metric measure space, let $\varnothing \neq \Omega\subseteq \XX$ be open and let $(\YY,\rho,\bar y)$ be a pointed separable metric space. For $u \in L^1_{\rm loc}(\Omega,\YY)$ and $\varnothing \neq U \subseteq \Omega$ open, we define
\[
V_{u}(U):= \inf \big\{ \liminf_{n\to\infty} |\DIFF u_n|(U)\colon (u_n) \subseteq  \mathcal S_{\rm loc}(U,\YY),\ \rho(u_n, u) \to 0\text{ in } L^1_{\rm loc}(U)  \big\}.
\]
We then say that $u \in \BV^*(\Omega,\YY_{\bar y})$, provided $u \in L^1(\Omega,\YY_{\bar y})$ and $V_u(\Omega)<\infty$. 

\end{defn}

\subsubsection{Variation of simple maps}\label{ss:var_simple}

Here we characterize the total variation of a simple map. In the general case faced in Lemma \ref{cdscascs}, {we will obtain effective bounds with structural constants depending only on the PI-parameters}. We will also include in Lemma \ref{lem:var simple isotropic} two closed-form formulas for the total variation of a simple map when the source space is assumed also isotropic or with the two-sideness property. These more intuitive formulas actually motivate our search for the general bounds obtained in Lemma \ref{cdscascs}.

\medskip

We start setting up some notation. For a partition $(E_i)_{i=1,..,n}$ of $\varnothing \neq \Omega\subseteq \XX$ open (in practice, it will always consists of sets of finite perimeter), we define
\[
\Gamma_S\coloneqq  { \Omega\cap}   \bigcap_{i\in S}\partial^*E_i\setminus  \bigcup_{i\notin S}\partial^*E_i \,\qquad \text{for every } \varnothing\neq S\subseteq\{1,\ldots,n\}. 
\]
Observe thus that $\{\Gamma_S\}_{S\subseteq \{1,\dots,n\}, \#S\ge 2}$ is a partition of ${ \Omega\cap}\bigcup_{i=1,\dots,n}\partial^*E_i $. To see this, call $\Gamma \coloneqq {\Omega\cap }\bigcup_{i=1,\dots,n}\partial^* E_i$ and notice the trivial inclusion $ \{\Gamma_S\}_{S\subseteq \{1,\dots,n\}, \#S\ge 2} \subseteq \Gamma$. For the converse, just observe that if $x \in \Gamma$, then there is $\tilde S\subseteq \{1,...,n\}$ with $\#\tilde S\ge 2$ (since $(E_i)$ is a partition of $\Omega$) so that $x \in \Gamma_{\tilde S}$. Finally, it can be readily checked that $\Gamma_{S_1} \cap \Gamma_{S_2} = \varnothing$ for $S_1 \neq S_2$, proving that it is indeed a partition.

\begin{lem}\label{cdscascs}
Let $(\XX,\dist,\mass)$ be a PI space, let $\varnothing\neq \Omega\subseteq\XX$ be open and let  $(\YY,\rho)$ be a separable metric space. Consider $u \in \SS(\Omega,\YY)$  with $|\DIFF u|(\Omega)<\infty$, say $u=\sum_{i=1}^n y_i\chi_{E_i}$, where \((y_i)_{i=1}^n\subseteq\YY\) and \((E_i)_{i=1}^n\) is a partition of \(\Omega\).  Then, for some $c>0$ depending only on the PI-parameters, it holds, as measures on $\Omega$, that
$$
c^{-1}\sum_{S\subseteq\{1,\dots,n\}, \#S\ge 2} \max_{i,j\in S}\rho(y_i,y_j)\HH^h\mres \Gamma_S\le |\DIFF u|\le c\sum_{S\subseteq\{1,\dots,n\}, \#S\ge 2} \max_{i,j\in S}\rho(y_i,y_j)\HH^h\mres \Gamma_S
$$
and
$$
c^{-1}\sum_{i\ne j} \rho(y_i,y_j)\HH^h\mres ({\Omega\cap }\partial^*{E_i}\cap\partial^*{E_j})\le |\DIFF u|\le c \sum_{i\ne j} \rho(y_i,y_j)\HH^h\mres ({ \Omega\cap}\partial^*{E_i}\cap\partial^*{E_j}).
$$
\end{lem}
\begin{proof}
Let us first make some preliminary observations. Fix an arbitrary set $S\subseteq\{1,\dots,n\}$ with $S\ge 2$. From Remark \ref{rem:simple map bv}, we know that $E_i$ are of finite perimeter and also that $|\DIFF u|$ is concentrated on $\Gamma ={\Omega\cap } \cup_{i=1,...,n}\partial^*E_i $. Thus, recalling that $\{\Gamma_S \colon S\subseteq \{1,...,n\}, \ \#S\ge 2\}$ is a partition of $\Gamma$, we have the identity:
\begin{equation}\label{eq:var simple GammaS}
|\DIFF u|=\sum_{\substack{S\subseteq\{1,\ldots,n\}: \\ \# S\geq 2}}|\DIFF u|\mres{\Gamma_S}.
\end{equation}
We claim now that it is enough to show as measures on $\Omega$ that
\begin{equation}
    c^{-1} \max_{i,j\in S}\rho(y_i,y_j)\HH^h\mres  \Gamma_S\le |\DIFF u|\mres \Gamma_S \le c \max_{i,j\in S}\rho(y_i,y_j)\HH^h\mres  \Gamma_S, \label{eq:claim tv simple}
\end{equation}
for every $S \subseteq \{1,...,n\}$ with $\#S \ge 2$. Indeed, admitting the above, let us estimate
\[\begin{split}
\sum_{i\ne j} \rho(y_i,y_j)\HH^h\mres (\partial^*{E_i}\cap\partial^*{E_j}\cap \Gamma_S)&=\sum_{i\ne j\text{ with } i,j\in S} \rho(y_i,y_j)\HH^h\mres (\partial^*{E_i}\cap\partial^*{E_j}\cap \Gamma_S)\\&
\le \sum_{i\ne j\text{ with } i,j\in S} \max_{h,k\in S}\rho(y_h,y_k)\HH^h\mres (\partial^*{E_i}\cap\partial^*{E_j}\cap \Gamma_S)
\\&\le c \max_{h,k\in S}\rho(y_h,y_k)\HH^h\mres  \Gamma_S
\end{split}\]
where in the last inequality we make use of Remark \ref{rem:S max}. Also, if  we take $i_0,j_0\in S$ satisfying $\max_{h,k\in S}\rho(y_h,y_k)=\rho(y_{i_0},y_{j_0})$, then
\begin{align*}
\max_{h,k\in S}\rho(y_h,y_k)\HH^h\mres  \Gamma_S = \rho(y_{i_0},y_{j_0})\HH^h\mres  \Gamma_S &= \rho(y_{i_0},y_{j_0})\HH^h\mres(\partial^* E_{i_0}\cap\partial^* E_{j_0}\cap \Gamma_S) \\
&\le \sum_{i\ne j} \rho(y_i,y_j)\HH^h\mres (\partial^*{E_i}\cap\partial^*{E_j}\cap \Gamma_S).
\end{align*}
Combining \eqref{eq:var simple GammaS}-\eqref{eq:claim tv simple} with the above, the proof would be concluded.

Therefore, we only need to prove now \eqref{eq:claim tv simple} and we can assume to work with \emph{real valued} simple functions, i.e.\ with $\YY=\RR$ (simply by taking the essential supremum over $1$-Lipschitz functions in \eqref{eq:claim tv simple} written for post-composed maps). Given \(S\subseteq\{1,\ldots,n\}\) with \(\# S\geq 2\), without loss of generality we can suppose $y_1\le y_2\le\dots\le y_n$.  Define \(F_i\coloneqq E_1\cup\dots\cup E_i\) and observe
\begin{equation}\label{eq:var_simple_isotropic_aux1}
\Gamma_S\subseteq{ \Omega\cap}\partial^* F_i \;\text{ if }\min(S)\leq i<\max(S),\qquad\Gamma_S\subseteq{ \Omega} \setminus\partial^*F_i\;\text{ otherwise.}
\end{equation}
To prove the former, notice that if \(m\coloneqq\min(S)\leq i<\max(S)\eqqcolon M\) and \(x\in\partial^*E_m\cap\partial^*E_M\), then
\(\overline D(F_i,x)\geq\overline D(E_m,x)>0\) and \(\overline D(\XX\setminus F_i,x)\geq\overline D(E_M,x)>0\), proving that
\(x\in\partial^*F_i\) and thus \(\Gamma_S\subseteq{ \Omega\cap}\partial^*E_m\cap\partial^*E_M  \subseteq{\Omega\cap }\partial^*F_i \). To prove the latter,
we fix any \(x\in\Gamma_S\) and distinguish two cases. If \(i<m\), then \(x\in{ \Omega} \setminus\bigcup_{j\leq i}\partial^*E_j\subseteq{ \Omega} \setminus\partial^*F_i\).
If \(i\geq M\), then
\[
x\in\bigcap_{j>i}{ (\Omega\setminus\partial^*E_j)}={ \Omega}\setminus\bigcup_{j>i}\partial^*E_j
\subseteq{ \Omega}\setminus\partial^*({ \Omega}\setminus F_i)={ \Omega}\setminus\partial^*F_i.
\]
Next, for any \(t\in\RR\) it holds that
\[
\{u<t\}=\left\{\begin{array}{ll}
\varnothing,\\
F_i,\\
\Omega,
\end{array}\quad\begin{array}{ll}
\text{ if }t\leq y_1,\\
\text{ if } y_i<t\leq y_{i+1}\text{ for some }i=1,\ldots,n-1,\\
\text{ if }t> y_n.
\end{array}\right.
\]
By using the coarea formula, we can thus write
\[
|\DIFF u|=\int_\RR {\rm Per}(\{u<t\},\cdot)\,\dd t=\int_{y_1}^{y_n}{\rm Per}(\{u<t\},\cdot)\,\dd t=\sum_{i=1}^{n-1}(y_{i+1}-y_i){\rm Per}(F_i,\cdot).
\]
Now, from \eqref{eq:var_simple_isotropic_aux1} and by representation formula, we have \({\rm Per}(F_i,\cdot)\mres{\Gamma_S}=\theta_{F_i}\HH^h\mres{\Gamma_S}\) if \(\min(S)\leq i<\max(S)\) and \({\rm Per}(F_i,\cdot)\mres{\Gamma_S}=0\)
otherwise, so that we can reach
    $$
    |\DIFF u|\mres\Gamma_S =\sum_{i=1}^{n-1}(y_{i+1}-y_i){\rm Per}(F_i,\cdot)\mres{\Gamma_S}=\sum_{i=\min(S)}^{\max(S)-1}(y_{i+1}-y_i)\theta_{F_i}\HH^h\mres \Gamma_S.
    $$
Finally, using that $\theta_{F_i}$ is bounded away from zero and from above uniformly in terms of structural constants (recall Theorem \ref{thm:repr_per}), we get
   $$
    c^{-1} (y_{\max(S)}-y_{\min(S)})\HH^h\mres \Gamma_S\le |\DIFF u|\mres\Gamma_S\le c(y_{\max(S)}-y_{\min(S)})\HH^h\mres \Gamma_S,
    $$
By our choice $y_{\max (S)}-y_{\min (S)} = \max_{i,j}|y_i-y_j|$, the above is exactly \eqref{eq:claim tv simple} for real-valued functions. From what was said before, this is enough to conclude the proof.
\end{proof}

\begin{lem}\label{lem:var simple isotropic}
Let \((\XX,\dist,\mass)\) be an isotropic PI space, let $\varnothing \neq \Omega \subseteq \XX$ be open and let $(\YY,\rho)$ be a separable metric space. Let \(u \in \mathcal{S}(\Omega,\YY)\) with $|\DIFF u|(\Omega)<\infty$, say
\(u=\sum_{i=1}^n y_i\nchi_{E_i}\), where \((y_i)_{i=1}^n\subseteq\YY\) and \((E_i)_{i=1}^n\) is a partition of \(\Omega\).
Then 
\begin{equation}\label{eq:var_simple_isotropic_claim}
|\DIFF u|=\sum_{\substack{S\subseteq\{1,\ldots,n\}: \\ \# S\geq 2}}\max_{i,j\in S}\rho(y_i,y_j)
\theta\HH^h\mres{ \Gamma_S},
\end{equation}
for some $\theta \colon \Omega \to[\gamma,C]$ Borel depending on $(E_i)$. If $\XX$ has the two-sidedness property, then
\begin{equation}\label{eq:var_simple_two-sided_claim}
|\DIFF u|=\frac{1}{2}\sum_{i,j=1}^n \rho(y_i,y_j)\theta\HH^h\mres{({\Omega\cap\partial^*E_i\cap\partial^*E_j })}.
\end{equation}
\end{lem}
\begin{proof}
It is enough to prove both the claims \eqref{eq:var_simple_isotropic_claim}-\eqref{eq:var_simple_two-sided_claim} for real-valued functions, i.e.\ with $\YY = \RR$. We can then suppose that $y_1\le y_2\le\dots\le y_n$ and, arguing exactly as in the proof of Lemma \ref{cdscascs}, we have (with the same notation)
   $$
    |\DIFF u|\mres\Gamma_S =\sum_{i=\min(S)}^{\max(S)-1}(y_{i+1}-y_i)\theta_{F_i}\HH^h\mres \Gamma_S.
    $$
Using now Corollary \ref{cor:isotrop_BV},
we have $\theta_{F_i}=\theta_{\min(S)}$ $\HH^h\mres{\Gamma_S}$-a.e.\ so that the above becomes
\[
 |\DIFF u|\mres\Gamma_S = (y_{\max(S)}-y_{\min(S)}) \theta\HH^h\mres{\Gamma_S} = \max_{i,j \in S}|y_i-y_j|\theta\HH^h\mres{\Gamma_S},
\]
for a Borel function $\theta\colon \Omega\to[\gamma,C]$ depending only on the sets $E_i$. This concludes the claim of \eqref{eq:var_simple_isotropic_claim} by arbitrariness of $S$.

In the case that $\XX$ satisfies also the two-sidedness property, then \eqref{eq:var_simple_two-sided_claim} follows directly from \eqref{eq:var_simple_isotropic_claim} and the following claim:

\begin{equation}\label{eq:var_simple_two-sided_aux}
S\subseteq\{1,\ldots,n\},\,\HH^h(\Gamma_S)>0\quad\Longrightarrow\quad\# S=2.
\end{equation}
Indeed, given \(i,j\in\{1,\ldots,n\}\) with \(i\neq j\), we have that \(\HH^h \)-a.e.\ \(x\in\partial^*E_i\cap\partial^*E_j\) does not
belong to \(\partial^*(E_i\cup E_j)\). For such a point \(x\), we have that \(\overline D(E_i\cup E_j,x)\geq\overline D(E_i,x)>0\) and
thus \(\overline D(\XX\setminus(E_i\cup E_j),x)=0\) as $x\notin \partial^*(E_i\cap E_j)$. In particular, for every \(k\in\{1,\ldots,n\}\setminus\{i,j\}\) we have that
\(\overline D(E_k,x)\leq\overline D(\XX\setminus(E_i\cup E_j),x)=0\) and thus \(x\notin\partial^*E_k\). This proves the validity of \eqref{eq:var_simple_two-sided_aux}.
    \end{proof}

\subsubsection{Variation measure}
Take $u\in L^1_{\mathrm{loc}}(\Omega,\YY)$, where $\varnothing \neq \Omega\subseteq\XX$ is open. We recall the definition of the variation $V_u$ for convenience and we also define an auxiliary variation via relaxation in $L^1$-topology: for every $U\subseteq\Omega$ open:
\begin{align*}
    &V_{u}(U) = \inf \big\{ \liminf_{n\to\infty} |\DIFF u_n|(U)\colon (u_n) \subseteq  \mathcal S_{\rm loc}(U,\YY),\ \rho(u_n,u)\rightarrow 0 \text{ in } L^1_{\rm loc}(U)  \big\},\\
    &\tilde V_{u}(U)\defeq \inf \big\{ \liminf_{n\to\infty} |\DIFF u_n|(U)\colon (u_n) \subseteq  \mathcal S_{\rm loc}(U,\YY),\ \rho(u_n,u)\rightarrow 0 \text{ in } L^1(U)  \big\}.
\end{align*}
We extend the above set-value map defined on open subsets to all Borel sets:
\begin{align*}
	\alpha(B) &\defeq \inf \{V_u ( U) \colon U \text{ open}, B \subseteq U \}, \qquad \forall B\subseteq \Omega \text{ Borel},\\
\tilde\alpha(B) &\defeq \inf \{\tilde V_u ( U) \colon U \text{ open}, B \subseteq U \}, \qquad \forall B\subseteq \Omega \text{ Borel}.
\end{align*}
Notice that $V_u\le \tilde V_u$, whence $\alpha\le \tilde\alpha$ and also that if $U\Subset V\subseteq\Omega$ are open, then $V_u(V)\ge  \tilde V_u(U)$ (we are assuming the space to be PI, hence proper).

\begin{thm}\label{thm:V is a measure}
	Let $(\XX,\dist,\mass)$ be a PI space, let $\varnothing\neq \Omega\subseteq\XX$ be open and let $(\YY,\rho)$ be a separable metric space. Fix $u\in L^1_{\mathrm{loc}}(\Omega,\YY)$ with $V_u(\Omega)<+\infty$. Then $\alpha=\tilde\alpha$ are Borel measures on $\Omega$. In particular, $V_u$ and $\tilde V_u$ coincide and are the trace on open sets of the same Borel measure on $\Omega$.
\end{thm}
	\begin{proof}
	It is enough to show that $\alpha$ and $\tilde \alpha$ are Borel measures on $\Omega$ with $\alpha(\Omega)=\tilde\alpha(\Omega)$.
	
	We argue using the well-known criterion \cite{DeGiorgiLetta77} (here we will follow \cite[Lemma 5.1]{AmbrosioDiMarino14}). We will prove that
	\begin{enumerate}[label=\roman*)]
		\item $V_u (U_1) \le V_u (U_2)$ for $U_1\subseteq U_2\subseteq\Omega$ open,
		\item $V_u(U_1 \cup U_2) \le V_u (U_1) + V_u (U_2)$ for $U_1,U_2 \subseteq \Omega$ open. Moreover,  equality holds if  $U_1$ and $U_2$ are well separated, i.e.\ {$\dist(U_1, U_2)>0$},
		\item if we have a sequence of open   sets $U_n\nearrow U\subseteq\Omega$, then $ V_u(U_n) \nearrow V_u( U)$.
	\end{enumerate}
We will prove similar properties for $\tilde V_u$. Property $\rm i)$ is trivial both for $V_u$ and $\tilde V_u$, as well as $\rm ii)$ in the case in which $U_1$ and $U_2$ are well separated. We will exploit these facts in what follows.
We address now property $\rm iii)$ both for $V_u,\tilde V_u$. We want to show that, if $U_n\nearrow U\subseteq\Omega$ are open sets, then
\begin{equation}\notag
	 V_u (U_n)\nearrow  V_u(U)\qquad\text{and} \qquad \tilde V_u (U_n)\nearrow  \tilde V_u(U).
\end{equation}
Notice that, by monotonicity with respect to inclusions, up to replacing $U_n$ with (here $\bar x\in\XX$ is any fixed point)
$$
\{x\in U_n\cap B_n(\bar x): \dist(x,\XX\setminus U_{n})>1/n\},
$$ 
there is no loss of generality in assuming that
$$
U_n\Subset U_{n+1}\qquad\text{for every }n.
$$
As $V_u(U_n)\le\tilde V_u(U_n)\le V_u(U_{n+1})$ and $V_u(U)\le \tilde V_u(U)$,  it will be enough to show that
\begin{equation}\label{bigclaimsdacas}
 \tilde V_u(U_n)\nearrow\tilde V_u(U),
\end{equation}
since this will also imply $V_u(U_n)\nearrow V_u(U)$ and {$V_u(U)=\tilde V_u(U)$}. 

We define 
\begin{equation}\notag
			S_k\defeq
	\begin{cases}
U_2\qquad&\text{if $k=1$},\\
U_k\setminus \overline{U_{k-2}}\qquad&\text{if $k\ge 2$}.
	\end{cases}
\end{equation}
Now notice that if $2\le n\le n+3\le m$, then $$\dist(S_n,S_m)=\dist(U_n\setminus\overline{U_{n-2}}, U_m\setminus\overline{U_{m-2}})\ge \dist(U_n,\XX\setminus  U_{n+1})>0.$$
Then, the  families $\{ S_{3k}\}_{k\ge 1}$, $\{ S_{3k+1}\}_{k\ge 1}$ and $\{S_{ 3k+2}\}_{k\ge 1}$ have well separated elements which are contained in $\Omega$. Therefore, by what we have proved so far, for any $\epsilon>0$ fixed, there exists $\bar k\in\NN$ such that 
$$
\sum_{k\ge \bar k} V_u( S_k)\le \epsilon.
$$
Now we build a sequence $(u_m)\subseteq  \mathcal S_{\rm loc}(U,\YY)$ with $\rho(u_m,u)\rightarrow 0$ in $L^1(U)$ satisfying
\begin{equation}
    \tilde V_u(U_{\bar k})+2\epsilon\ge\liminf_m |\DIFF u_m|(U) \label{competitor}
\end{equation}
and this will prove \eqref{bigclaimsdacas}, being $\epsilon>0$ arbitrary.

Fix now $m \in \NN$ and, for every $h \in \NN$, define $D_h \coloneqq C_{h + \bar k}, B_h \coloneqq U_{h +\bar k}$ if $h\ge 1$ and $B_0 = D_0 = U_{\bar k}$. For every $h\ge 0$, by definition of $\tilde V_u(D_h)$, we can take $(\psi_{k,h})_k \subseteq \SS_{\rm loc}(D_h,\YY)$ such that
\begin{equation}
     |\DIFF \psi_{k,h}|(D_h) \le \tilde V_u(D_h) + \frac{1}{m2^k}.\label{eq:1}
\end{equation}

{Up to considering} a sufficiently large $k(h)\ge h$ depending on $h$, since $\rho(\psi_{k,h}, u) \to 0$ in $L^1(D_h)$, we can also require
\begin{equation}
    \int_{D_h} \rho(\psi_{k(h),h},u)\,\dd\mass \le \frac{1}{m2^{h}}.\label{eq:2}
\end{equation}
Define now $\eta_h \coloneqq \dist( D_{h+1}\setminus B_h, B_h \setminus D_{h+1})>0$. Since $B_h \cap D_{h+1} \subseteq  D_h$, {possibly enlarging} $k(h)$, we can guarantee also that
\begin{equation}
    \frac{C}{\eta_h}\int_{ D_h} \rho(\psi_{k(h),h)}, u)\,\dd\mass \le \frac{\epsilon}{2\cdot 2^h},\qquad   \frac{C}{\eta_h}\int_{ D_{h+1}}\rho(\psi_{k(h+1),h+1}, u)\,\dd\mass \le \frac{\epsilon}{2\cdot 2^h}. \label{eq:3}
\end{equation}
where $C>0$ is the constant depending only on the doubling and Poincar\'e parameters appearing in \eqref{casdcsaasca}.

Set now $u_{m,0}\coloneqq \psi_{k(0),0} $ on $ D_0$, i.e.\ so that $u_{m,0} \in \SS( B_0 ,\YY)$. Following exactly now the proof of \cite[Lemma 5.2]{AmbrosioDiMarino14}, a sequence $(u_m)\subseteq \SS_{\rm loc}(U,\YY)$ can be built by induction argument satisfying \eqref{competitor}. Just use \eqref{eq:1}-\eqref{eq:2}-\eqref{eq:3} in place of \cite[Eq. (5.6)-(5.8)-(5.7)]{AmbrosioDiMarino14} and the joint property in Lemma \ref{lem:joint} below in place of \cite[Lemma 5.4]{AmbrosioDiMarino14}

We thus conclude the proof by showing ii) in the general case. We start with $\tilde V_u$, and then the result for $V_u$ will follow easily. Taking into account iii) that we just proved, it is clearly sufficient to show that
\[
\tilde V_u(U_1' \cup U_2')\le \tilde V_u(U_1) + \tilde V_u(U_2), \qquad \forall U_1'\Subset U_1,U_2'\Subset U_2.
\]
We claim that this is indeed the case. Let $(u_{i,m}) \subseteq \SS_{\rm loc}(U_i,\YY)$ so that $\rho(u_{i,m},u)\to 0$ in $L^1(U_i)$ and $|\DIFF u_{i,m}|\to \tilde V_u(U_i)$, for $i=1,2$.  Then, for every $m \in \NN$, invoke Lemma \ref{lem:joint} below with $M \coloneqq (U_1'\cup U'_2)\cap U_1, N=(U_1'\cup U_2')\cap U_2$ and $u\coloneqq u_{1,m},v\coloneqq u_{2,m}$ to get $w_m \in \SS_{\rm loc}(U_1' \cup U_2',\YY)$ satisfying
\begin{align*}
 |\DIFF w_m|(U_1' \cup U_2') &\le |\DIFF u_{1,m}|(M) + |\DIFF u_{2,m}|(N) + \frac{C}{\dist(M\setminus N,N\setminus M)} \int_{M\cap N}\rho( u_{1,m},u_{2,m})\,\dd\mass \\   &\le |\DIFF u_{1,m}|(U_1) + |\DIFF u_{2,m}|(U_2) + \frac{C}{\dist(M\setminus N,N\setminus M)} \int_{U_1\cap U_2}\rho( u_{1,m},u_{2,m})\,\dd\mass.
\end{align*}
Then, by taking $m \nearrow \infty$, the claim follows. We thus proved that $\tilde V_u$ is the trace on open sets of the Borel measure $\tilde \alpha$. Finally, we argue for $V_u$: let us estimate
\[
V_u(U'_1\cup U'_2)\le \tilde \alpha(U_1'\cup U_2')\le \tilde \alpha(U'_1)+ \tilde V_u(U_2') \le V_u(U_1)+V_u(U_2)
\]
As before, this shows that $V_u$ is the trace on open sets of a Borel measure $\alpha$. The proof is now concluded.
\end{proof}

We now prove the result that allows us to join (locally) simple maps and has been heavily used in the proof of Theorem \ref{thm:V is a measure}.

\begin{lem}[Joint property]\label{lem:joint}
	
	Let $(\XX,\dist,\mass)$ be a PI space, let $M,N \subseteq \XX$ be  open sets such that $\dist(M\setminus N,N\setminus M)>0$ and let $(\YY,\rho)$ be a separable metric space. Let $ u\in \SS_{\rm loc}(M,\YY), v\in \SS_{\rm loc}(N,\YY)$ be two  locally simple maps. Then, there exists a  locally simple map $w \in \SS_{\rm loc}( M\cup N,\YY)$ satisfying
	\begin{equation}\label{casdcsaasca}
		|\DIFF w |( M\cup N) \le |\DIFF u|(M) + |\DIFF v|(N) + \frac{C}{\dist(M\setminus N,N\setminus M)}\int_{ M\cap N} \rho(u,v)\,\dd \mass,
	\end{equation}
where the constant $C>0$ depends only on $(\XX,\dist,\mass)$.
	Moreover,
 \begin{equation}\label{eqcoincide}
     \begin{split}
         &w=u\qquad\text{on a neighbourhood of $M\setminus N$},\\
         &w=v\qquad\text{on a neighbourhood of $N\setminus M$},
     \end{split}
 \end{equation}
 and, for every $\sigma \in L^1_{\rm loc}(M\cup N,\YY)$,
	\begin{equation}
		\int_{M\cup N} \rho(w,\sigma)\,\dd\mass \le \int_{M} \rho(u,\sigma)\,\dd\mass + \int_{N} \rho(v,\sigma)\,\dd\mass. \label{eq:stima joint L1}
	\end{equation}
\end{lem}

\begin{proof}
        If the right-hand side of \eqref{casdcsaasca} is $+\infty$, the claim is trivial. So, let us assume that $u\in \BV_{\rm loc}(M,\YY)$ and $v \in \BV_{\rm loc}(N,\YY)$ with $|\DIFF u|(M),|\DIFF v|(N)<\infty$. Let us set $\Omega\defeq M\cup N$, we can clearly assume that $M\cap N$ is non empty and that $\int_{M\cap N}\rho(u,v)\dd\mass<\infty$.
	Consider writing $u = \sum_i y_i\nchi_{E_i}, v = \sum_h z_h \nchi_{F_h}$ for $(y_i),(z_j) \subseteq \YY$ and for disjoint (minimal) partition $(E_i),(F_h)$ of $M,N$ respectively. 
	Here we have taken the measure theoretic interiors of the sets, i.e.\ we are taking $E_i=E_i^{(1)}$ and  $F_h=F_h^{(1)}$ for every $i,h$. 
	We fix these Borel representatives for $u$ and $v$ in what follows and the Borel representative for $\rho(u,v)$ will be fixed accordingly.  Notice that for every $\Omega'\Subset\Omega$, only finitely many sets in $(E_i),(F_h)$ intersect $\Omega'$.
	
	Let us denote $\eta \coloneqq \dist(M\setminus N,N\setminus M)$, notice $\eta>0$ and set
	\[
	D(\,\cdot \,) \coloneqq \dist(\cdot, N\setminus M).
	\]
	Denote $E_t \coloneqq \{ x \in \XX \colon D(x) < t \} $.  Notice that $\partial E_t\subseteq\{x\in\XX:\dist(x,N\setminus M)=t\}$ and that, for $t\in (0,\eta)$, $\Omega \cap E_t\subseteq N$, $\Omega\setminus E_t\subseteq M$ and $\Omega\cap \partial E_t\subseteq M\cap N$.
	Now we take $ s \in (0,\eta)$ so that:
	\begin{itemize}
		\item[i)] $E_s$ is of finite perimeter in $\Omega$.
		\item[ii)] $|\DIFF\chi_{E_i}|(\Omega\cap\partial E_s ) = 0$ for every $i$ and  $|\DIFF \chi_{F_h}|(\Omega\cap\partial E_s) = 0$ for every $h$.
		\item[iii)] $\int_{M\cap N} \rho(u,v)\,\dd \per(E_s,\,\cdot \,) \le \dashint_0^\eta \int_{M \cap N} \rho(u,v)\,\dd \per(E_r,\,\cdot \,)\dd r $.
	\end{itemize}
	We show now why this is possible. By the coarea formula a.e.\ $s\in (0,\eta)$ satisfies item $\rm i)$, whereas, as  $\partial E_s\cap\partial E_t=\varnothing$ for $s\ne t$, all but countably many $s\in (0,\eta)$ satisfy item $\rm ii)$. Finally, item $\rm iii)$ is true for $s $ in a set of positive measure in $(0,\eta)$.
	
	Now we define 
	\[
	w \coloneqq \begin{cases}
		\ u & \text{on }    \Omega\setminus E_s,\\
		\ v & \text{on } \Omega\cap E_s.
	\end{cases}
	\]
	Then $w \in \SS_{\rm loc}(\Omega,Y)$ and \eqref{eqcoincide} follows.
 Now, for every $A\Subset\Omega$, $w\in \SS(A,\YY)$ and we can compute by locality,
	\begin{equation}\label{csdcascdsaa}
		\begin{split}
			|\DIFF w|(A)&=|\DIFF w|(A\setminus E_s)+|\DIFF w|(A\cap E_s)+|\DIFF w|(A\cap\partial E_s)\\&\le |\DIFF u|(A\cap M)+|\DIFF v|(A\cap N)+|\DIFF w|(A\cap\partial E_s).
		\end{split}
	\end{equation}
Now we claim that for every open set $A\Subset\Omega$,
	\begin{equation}\label{cdsacdscdsascsd}
		|\DIFF w|(A\cap \partial E_s)\le C \int_{M\cap N}\rho(u,v)\dd\per(E_s,\,\cdot\,),
	\end{equation}
where $C>0$ depends only on $(\XX,\dist,\mass)$. Let us first show how to deduce \eqref{casdcsaasca} from  \eqref{cdsacdscdsascsd}. First, by \eqref{csdcascdsaa}, for every open set $A\Subset\Omega$,
	\begin{align*}
		|\DIFF w|(A)&\le |\DIFF u|( M)+|\DIFF v|( N)+\frac{C}{\eta} \int_{M\cap N}\rho(u,v)\dd\mass,
	\end{align*}
by item $\rm iii)$ for the choice of $s$ and  the coarea formula. Hence (as the increasing limit of measures is still a measure) there exists a finite measure $\mu$ with $\mu(\Omega)$ smaller or  equal to the right-hand side of \eqref{casdcsaasca} such that for every  open set $A\Subset\Omega$,
$|\DIFF w|\mres A\le \mu$. This easily imply, by the theory of \emph{real valued} functions of bounded variation,  that $w\in \BV_{\rm loc}(\Omega,\YY)$ with $|\DIFF w|(\Omega)\le \mu(\Omega)<\infty$.
	
	Now we prove \eqref{cdsacdscdsascsd}. Fix an open set $A\Subset \Omega$.  Notice that on $A$, $u,v,w$ are simple maps of bounded variation, hence we can assume, for this paragraph, that $(E_i)$ and $(F_h)$ are finite collections of sets.
	Now take, for some $i,h$, $x\in \partial ^*(E_i\setminus E_s)\cap \partial^* (F_h\cap E_s)\setminus (\partial ^* E_i\cup \partial^* F_h)$. Then $x\in\partial^* E_s\cap E_i\cap F_h$. Recall item $\rm ii)$ for the choice of $s$, then
	$\HH^h(A\cap\partial E_s\cap\partial^* E_i)=\HH^h(A\cap\partial E_s\cap\partial^* F_h)=0$ for every $i,h$. Hence we use Lemma \ref{cdscascs} to deduce that 	
	\begin{equation}\notag
		\begin{split}
			|\DIFF w|(A\cap \partial E_s)&\le C\sum_{i,h} \rho(y_i,z_h) \HH^h\mres (A \cap\partial^* E_s\cap E_i\cap F_h)\\&\le C\sum_{i,h} \chi_{E_i}\chi_{F_h}\rho(y_i,z_h) \theta_{E_s}\HH^h\mres (A \cap\partial^* E_s)
			\\&= C\int_{M\cap N}\rho(u,v)\dd\per(E_s,\,\cdot\,),
		\end{split}
	\end{equation}
	 which is \eqref{cdsacdscdsascsd}.  	Finally, for every $\sigma \in L^1(M\cup N,\YY)$, we have
	\begin{align*}
		\int_{M\cup N}\rho(w,\sigma)\dd\mass=\int_{\Omega\setminus E_s}\rho(w,\sigma)\dd\mass+\int_{\Omega \cap E_s}\rho(w,\sigma)\dd\mass,
	\end{align*}
whence \eqref{eq:stima joint L1} follows from the definition of $w$.
\end{proof}

\subsection{Weak duality with test plans}
We define a weak notion of maps of bounded variation (see \cite{AmbrosioDiMarino14} for the real-valued case).
\begin{defn}\label{def:Du weak}
Let $(\XX,\dist,\mass)$ be a metric measure space, let \(\varnothing\neq\Omega\subseteq\XX\) be open and let $(\YY,\rho,\bar y)$ be a pointed separable metric space.
Consider  $u \in L^1(\Omega,\YY_{\bar y})$. We say that $u \in \BV_{w}(\Omega,\YY_{\bar y})$ provided
there exists a finite Borel measure $\mu$ on $\Omega$ such that the following hold: given any $\infty$-test plan $\pi \in \mathscr{P}(C([0,1],\XX))$ satisfying
\(\gamma([0,1])\subseteq\Omega\) for \(\pi\)-a.e.\ \(\gamma\), we have
\begin{itemize}
    \item[{\rm i)}] $u \circ \gamma \in \BV([0,1],\YY)$ for $\pi$-a.e.\ $\gamma$;
    \item[{\rm ii)}] it holds that
\[
\int \gamma_*|\DIFF (u\circ \gamma)|(B)\dd\pi(\gamma)\le {\mathrm {Comp}}(\pi) {\mathrm {Lip}}(\pi)\mu(B),\qquad \forall B\subseteq \Omega\text{ Borel}.
\]
\end{itemize}
The minimal measure $\mu$ satisfying the above is denoted $|\DIFF u|_w$.
\end{defn}
We give two remarks on the above definition.
\begin{rem} 
The above definition is well-posed, thanks to the following facts. Notice that for $\pi$-a.e.\ $\gamma$ it holds $u \circ \gamma$ is independent of the chosen Borel representative of $u$. This is standard and due to the bounded compression of test plans. Also, we remark that $B \mapsto \int \gamma_*|\DIFF(u\circ \gamma)|(B)\,\dd\pi(\gamma)$ is a well defined Borel measure. This follows by standard arguments (see, e.g., \cite[Remark 2.8]{NobiliPasqualettoSchultz21} and recall Definition \ref{def:BV curve}).
\end{rem}

\begin{rem}\label{rem:weak BV on bounded plans}
Let us call an $\infty$-test plan $\pi$ \emph{bounded}, provided $\{ \gamma_t \colon \gamma \in \text{supp}(\pi), t \in [0,1]\} \Subset \Omega$. We claim that condition ii) in Definition \ref{def:Du weak} is equivalent to
\begin{itemize}
\item[ii')] there is a finite Borel measure $\mu$ so that, for all \emph{bounded} $\infty$-test plans $\pi$, it holds
\[
\int \gamma_*|\DIFF (u\circ \gamma)|(B)\dd\pi(\gamma)\le {\mathrm {Comp}}(\pi) {\mathrm {Lip}}(\pi)\mu(B),\qquad \forall B\subseteq \XX\text{ Borel}.
\]
\end{itemize}
Clearly, we only need to show that ii') implies ii). If $\Omega$ is bounded, there is nothing to prove. Otherwise, we argue as follows: given an arbitrary $\infty$-test plan $\pi$ with $\gamma([0,1])\subseteq \Omega$, consider e.g.\ for every $n\in\NN$ a bounded  $\infty$-test plan $\pi_n \coloneqq (\pi(\Gamma_n))^{-1}\pi\mres{\Gamma_n}$ where  $\Gamma_n \coloneqq \big\{ \gamma \in {\rm LIP}([0,1],\XX) \colon \dist(\gamma_0,\bar x) \le n,\,{\rm Lip}(\gamma)\le n\big\} \subseteq C([0,1],\XX)$ having also fixed  $\bar x \in \XX$. The conclusion follows by approximation. We omit the standard details.
\fr\end{rem}

\section{Equivalent characterizations}\label{sec:equivalence}
In this part, we prove our main result Theorem \ref{thm:main result 2}. This will be a consequence of a chain of inclusions between the spaces $\BV$, $\BV_w$,  $\BV^*$ that will be proved in this section separately. 

\subsection{Technical lemmas}

\begin{lem}\label{threedottwo}
	Let $(\XX,\dist,\mass)$ be a locally uniformly doubling space and let $f \in L^1_{\rm loc}(\XX)$. Then, for every $\delta>0$ and any couple of balls $B_r(x)$ and $B_s(z)$ with $r,s,\dist(x,z)\le \delta/3$ with $d s \ge \delta$ for some $d \in \NN$, it holds  that
	\[
	    \abs{\dashint_{B_r(x)}f\,\dd\mass-\dashint_{B_s(z)} f\,\dd\mass}\le   \frac{C( d \delta)}{\mass(B_{r}(x))}  \inf_{c\in\RR}\int_{B_{\delta}(x)}  |f - c|\,\dd\mass, 
    \]
    for some constant $C(d\delta)$ depending only on the local doubling constant ${\sf Doub}(d\delta)$.
\end{lem}
\begin{proof}
We use for brevity the notation $f_{x,r} \coloneqq \dashint_{B_r(x)}f\,\dd\mass $ in the following computations. By assumptions, $B_s(z),B_{r}(x)\subseteq B_{\delta}(x) \subseteq B_{d s}(z)$ so that, by doubling estimate, we get
\[ \mass(B_{s}(z))^{-1}\le C(d\delta)\mass(B_{ds}(z))^{-1}  \le C(d\delta) \mass(B_{\delta}(x))^{-1}.
\]
Therefore, for $c\in \RR$  arbitrary, we can estimate
\[
	    \begin{split}
	        | f_{x,r} - f_{z,s}| &\le \frac{1}{\mass(B_s(z))}\int_{B_{\delta}(x)} |f - f_{x,r}|\,\dd\mass  \le C( d \delta)\dashint_{B_{\delta}(x)} |f - f_{x,r}|\,\dd\mass  \\
	        &\le C( d \delta)\dashint_{B_{\delta}(x)} \dashint_{B_{r}(x)} |f(y_1)-f(y_2)|\,\dd\mass(y_1)\dd\mass(y_2) \\
	        &\le C( d \delta)\frac{\mass({B_{\delta}(x)) }}{\mass(B_{r}(x))}  \dashint_{B_{\delta}(x)} \dashint_{B_{\delta}(x)} |f(y_1)  - f(y_2)|\,\dd\mass(y_1)\dd\mass(y_2) \\
	        &\le   \frac{C( d \delta)}{\mass(B_{r}(x))}   \int_{B_{\delta}(x)}  |f - c|\,\dd\mass.    
	    \end{split}
\]
This concludes the proof.
\end{proof}
The following is a standard covering construction. We include a proof to be self-contained and as we are going to require also a couple of non-completely standard properties.

\begin{lem}\label{lem:covering}
Let $(\XX,\dist,\mass)$ be a uniformly locally doubling metric measure space, let $\varnothing \neq \Omega \subseteq \XX$ be open and let $\Lambda\ge 1$. Then, there exists a constant $C>0$, depending only on $(\XX,\dist,\mass)$ and $\Lambda$, such that for every $h\ge 3$, we can find a locally finite covering of $\Omega$, $\{B_i=B_{r_i}(x_i)\}_{i\in\NN}$, with the following properties, for every $i,j\in\NN$:
\begin{enumerate}
	\item $r_i\in(0,h^{-1})$;
	\item it holds \begin{equation}\label{sacdsac}
	    \nchi_\Omega\le\sum_i \nchi_{ B_i}\le\sum_i \nchi_{\Lambda B_i}\le C\nchi_\Omega\qquad\text{on }\XX,
	\end{equation} in particular, $\Lambda B_i\subseteq\Omega$;
	\item if $\bar B_i\cap \bar B_j\ne\varnothing$, then $r_j\le4 r_i$;
	\item $r_i \per(B_i,\XX)\le C \mass (B_i)$;
 \item $\HH^h(\partial B_i\cap\partial B_j)=0$.
\end{enumerate}
\end{lem}
\begin{proof}
For every $x\in \Omega$, define $$ r'_x\defeq  (2 h)^{-1}\min\bigg\{1,\frac{\dist(x,\XX\setminus\Omega)}{\Lambda}\bigg\},$$
where, if $\Omega=\XX$, we understand $\dist(x,\XX\setminus\Omega)=+\infty$.
We apply Vitali $5r$-covering lemma to obtain a sequence of balls $\{B'_i=B_{r_i'}(x_i)\}_{i\in\NN}$ with the following properties: $(1/5)B'_i$ are pairwise disjoint and $\{B'_i\}_i$ is a covering of $\Omega$ and such that for every $i$, there exists $x\in\Omega$  (depending on $i$) with $x_i=x$ and $r_i'=r'_x=r'_{x_i}$.
Now, for every $i$, define 
$$A_i\defeq\big\{r\in  (r'_i, 2r'_i) : r'_i\per (B_{r}(x_i),\XX)\le \mass(B_{2 r'_i}(x_i))\big\},
$$
notice that $\mathcal{L}^1(A_i)>0$ by coarea.
Now we define, by induction, $\{r_i\}_i$. First, take $r_1\in A_1$. For $i>1$, choose $r_i\in A_i$ in such a way that 
$$\HH^h\big(\partial B_{ r_j}(x_j)\cap\partial B_{ r_i}(x_i)\big)=0\qquad\text{for every }j=1,\dots,i-1.$$
Define then $B_i\defeq B_{r_i}(x_i)$ and notice that $\frac{1}{10}B_i$ are disjoint by construction. We claim that the family $\{B_i\}_i$ satisfies the properties of the claim.

Properties $1,4$ and $5$ are trivially satisfied by the construction used, taking into account also the doubling inequality. 
Take now $B_i$ and $B_j$ with $\dist(x_i,x_j)\le \Lambda(r_i+r_j)$. We claim that 
$r_j\le 4 r_i$. 
Indeed, if $r_i'=(2 h)^{-1}$, then $r_j\le 2 r_i \le 4r_i$.
Otherwise (hence $\Omega\ne\XX$), $$(2h)^{-1} \frac{\dist(x_i,\XX\setminus\Omega)}{\Lambda}=r_i'\le r_i$$ so that (as $h\ge 3$)
$$
r_j\le 2r_j'\le 2(2 h)^{-1} \frac{\dist(x_j,\XX\setminus\Omega)}{\Lambda}\le  \frac{\dist(x_j,x_i)}{h\Lambda}+  \frac{\dist(x_i,\XX\setminus\Omega)}{h \Lambda}\le \frac{r_i+r_j}{3}+ 2{r_i},
$$
giving in turn $r_j \le 7/2r_i\le 4r_i$ and proving the claim. From this remark, property $3$ follows and also property $2$ follows, taking into account standard arguments involving the doubling property of measures. We briefly sketch the argument leading to the second inequality in \eqref{sacdsac}, $C$ will be a constant as in the statement and may vary during the proof. Notice that we are going to prove a slightly stronger statement, as, in this way, we obtain as a byproduct that the covering is locally finite.
More precisely, we are going to show that for every $i_0\in\mathbb N$, setting
$$J_{i_0}\defeq\{i: \Lambda B_i\cap \Lambda B_{i_0}\ne\varnothing\},$$
it holds that $\#J_{i_0}\le C$. Take any $i\in J_{i_0}$, then $\dist(x_{i_0},x_i)\le \Lambda(r_{i_0}+r_i)$, hence, by what proved above, $$r_i\le 4 r_{ i_0}\qquad\text{and}\qquad r_{ i_0}\le 4 r_i.$$ 
In particular, $\frac{1}{10}B_i\subseteq (2/5+5\Lambda) B_{i_0}$ and $B_{i_0}\subseteq 10(4+5\Lambda) \frac{1}{10}B_i$. Being $\{\frac{1}{10}B_i\}_{i\in J_{i_0}}$ disjoint by construction, the claim follows by the doubling property of the measure. 
\end{proof}

\subsection{Proof of equivalences}
\begin{prop}\label{thm:Equivalent BVmaps PI}
	Let $(\XX,\dist,\mass)$ be a PI space, let $(\YY,\rho,\bar y)$ be a pointed separable metric space and consider $u \in L^1(\XX,\YY_{\bar y})$. The following are equivalent:
	\begin{itemize}
	    \item[{\rm i)}] $u \in \BV(\XX,\YY_{\bar y})$;
	    \item[{\rm ii)}] there is $A\subseteq\XX$ with $\mass(\XX\setminus A)=0$, a constant $\sigma\ge 1$, a finite non-negative Borel measure $\mu$ and, for every $R>0$, there is $C(R)>0$ such that
	    \begin{equation}
	       \rho(u(x),u(y))\le C(R)\dist(x,y)\big(\mathcal{M}_{\sigma  \dist(x,y)}(\mu)(x) +\mathcal{M}_{\sigma  \dist(x,y)}(\mu)(y) \big),\quad
        \begin{array}{c}
             \forall x,y\in A,  \\
             \dist(x,y)\le R;
        \end{array}
	    \label{eq:maximal estimate}
	    \end{equation}
	    
	    \item[{\rm iii)}] there is a finite non-negative Borel measure $\nu$, $\lambda\ge 1$ and, for every $R>0$, there is $C(R)>0$ so that: for every ball $B_r(x)$ with $r \in (0,R)$, there is a point $y\in\YY$ s.t.
		\begin{equation}\label{fakepoincare}
			\int_{B_r(x)} \rho(u,y)\dd\mass\le  C(R)r \nu(B_{\lambda r}(x)).
		\end{equation}
	\end{itemize}
	Moreover, if any of the above holds true, in {\rm ii)}, {\rm iii)} we can take $\mu,\nu= |\DIFF u|$, $\sigma =2\tau_p,\lambda = 6\tau_p$ and the constants $C(R)$  can be taken depending only on $R$ and the PI-parameters.
	\end{prop}
	\begin{proof}
    In this proof, a constant $C(R)>0$ can change from line to line but depending only on $R>0$ and the PI-parameters.
    \medskip\\\textbf{Implication} i) $\Rightarrow$ ii). The implication is known if $\YY = \RR$ and $\XX$ is globally doubling (recall, we only assume the space to be locally uniformly doubling); see, e.g., \cite{LTpointwise} relying on the argument in \cite[Theorem 3.2]{HajlaszKoskela00}. 

    In the general case, we include a quick proof. Let $\varphi \in  \LIP(\YY)$ be any $1$-Lipschitz function with $\varphi(\bar y)=0$ and set $f \coloneqq \varphi \circ u \in \BV(\XX)$ by assumption. Consider $x,y \in \XX$ Lebesgue points of $f$ and let $R \ge \dist(x,y)$. Define $r_i\coloneqq 2^{-i}\dist(x,y)$ and $f_{x,r_i}\coloneqq \dashint_{B_{r_i}(x)} f\,\dd\mass $ for brevity, for all $i=0,1,..$, so that $f_{x,r_i} \to f(x)$ as $i\nearrow\infty$.
    Then
    \[\begin{split}
            |f(x)-f_{x,r_0}| &\le \sum_{i=0}^\infty |f_{x,r_{i+1}}-f_{x,r_i}| = \sum_{i=0}^\infty\dashint_{B_{r_{i+1}}(x)}|f - f_{x,r_i}|\,\dd \mass \\
            &\overset{\eqref{eq:doubling}}{\le} \frac{1}{{\sf Doub}(R)} \sum_{i=0}^\infty\dashint_{B_{r_i(x)}}|f - f_{x,r_i}|\,\dd \mass \overset{\eqref{eq:Poincare BV}}{\le} \frac{C(R) }{{\sf Doub}(R)} \sum_{i=0}^{\infty} r_i\frac{|\DIFF f|(B_{\tau_p r_i}(x))}{\mass(B_{r_i}(x))} \\
            &\le C(R) \dist(x,y) \mathcal{M}_{\tau_p \dist(x,y)}(|\DIFF u|)(x),
    \end{split}
    \]
    holds where $\tau_p\ge 1$ given in \eqref{eq:Poincare BV}. A symmetric estimate holds for the term $|f(y)-f_{y,r_0}|$. Moreover, since $B_{r_0}(y) \subseteq B_{2r_0}(x)$, we can estimate
    \[\begin{split}
        |f_{x,r_0}-f_{y,r_0}| &\le |f_{x,r_0}-f_{x,2r_0}| + |f_{x,2r_0}-f_{y,r_0}| \\
        &\le \dashint_{B_{r_0}(x)} |f- f_{x,2r_0}|\,\dd\mass +  \dashint_{B_{r_0}(y)} |f- f_{x,2r_0}|\,\dd\mass \le  C(R)\dashint_{B_{r_0}(x)} |f - f_{x,2r_0}|\,\dd\mass\\
        &\overset{\eqref{eq:Poincare BV}}{\le}  C(R)\dist(x,y) \frac{|\DIFF f|(B_{2\tau_p \dist(x,y)}(x))}{\mass(B_{2\dist(x,y)}(x))} \le C(R)\dist(x,y)  \mathcal{M}_{2\tau_p \dist(x,y)}(|\DIFF u|)(x).
    \end{split}
    \]
  Notice that, thanks to the Lebesgue differentiation theorem in PI spaces, $\mass$-a.e.\ point is a Lebesgue point of  $f \in L^1(\XX)$ and thus is suitable for the above estimates to hold. Therefore, we can select a set $A_\varphi \subseteq \XX$ of full measure satisfying
    \[
	     |\varphi \circ u(x)-\varphi \circ u (y)|\le C(R)\dist(x,y)\big( \mathcal{M}_{2\tau_p \dist(x,y)}(|\DIFF u|)(x) {+}  \mathcal{M}_{2\tau_p \dist(x,y)}(|\DIFF f|)(y)\big),\quad \begin{array}{l}
	          \forall x,y\in A_{\varphi},  \\
	           \dist(x,y)\le R.
	     \end{array}
    \]
    Now, consider  $\varphi_n = \rho(\cdot, y_n)$ for $(y_n)_n\subseteq\YY$ countable dense and set $A\coloneqq \cap_{\varphi_n}A_{\varphi_n}$ that has negligible complement. Taking the supremum in the above, we reach \eqref{eq:maximal estimate} with a suitable $C(R)$ and for $\sigma=2\tau_p$.
\medskip\\\textbf{Implication} ii) $\Rightarrow$ iii). We closely follow \cite[Theorem 3.2]{LTpointwise} which deals with globally doubling spaces and real-valued functions. We shall only highlight the differences in the arguments, the spirit of the proof being the same.
	
    Given $R>0$ and $\sigma>0$ in ii), we claim that, for a ball $B_r(x)$ with $r\le R/2$, \eqref{fakepoincare} holds true for some point $y \in \YY$ to be chosen, $\lambda \coloneqq 3\sigma$, a suitable constant $C(R)>0$ depending only on $R,{\sf Doub}(R), C_p(R)$  and for $\nu = \mu$.

    In the case $ \mu(B_{3\sigma r}(x) )=0$, then this claim is trivial. Indeed, any couple $x_1,x_2 \in A \cap B_{r}(x) $ is so that $\rho(u(x_1),u(x_2))= 0$, since $\mathcal{M}_{\sigma\dist(x_1,x_2)}(\mu) \le \mathcal{M}_{3\sigma r}(\mu) =0$. In particular, $u$ is $\mass$-a.e.\ constant in $B_r(x)$ and thus, a suitable $y  \in \YY$ exists giving zero in the left-hand side of \eqref{fakepoincare}. 

    We pass to the verification of the general case and assume $ \mu(B_{3\sigma r}(x) )>0$. We follow the proof of \cite[Theorem 3.2]{LTpointwise}  (see also \cite{Haj03,HajNew}) considering the function $\rho(u(\,\cdot\,),y)$ for some suitable $y\in\YY$.  Instead of concluding with the same claim of  \cite[Theorem 3.2]{LTpointwise}, it will be enough for us to prove the second inequality of  \cite[Equation (3.4)]{LTpointwise}. Choose $k_0\in\ZZ$ as in the proof of \cite[Theorem 3.2]{LTpointwise}.  We choose $y\in\YY$ such that $\essinf_{E_{k_0}}\rho(u(\,\cdot\,),y)=0$, where {
    $$E_{k_0}\defeq \{x'\in B_{r}(x):\mathcal M(\mu\mres{3\sigma B})(x')\le 2^{k_{0}}\}.$$
    }
    We show now why this is possible. 	By the proof of \cite[Theorem 3.2]{LTpointwise}, we know that $\mass(E_{k_0})>0$ and $u$ is $C(R)2^{k_0+1}$-Lipschitz on $E_{k_0}\cap A$. Fix $\epsilon>0$; for every $\rho \le \frac{\epsilon}{C(R)2^{k_0+1}}$ we thus have
    \[
    \{ x' \in E_{k_0} \colon \rho(u(z),u(x'))\le \epsilon \} \cap B_{\rho}(z)\supseteq E_{k_0} \cap B_{\rho}(z),\qquad \mass\text{-a.e.\ }z \in E_{k_0}.   
    \]
    From this and the Lebesgue differentiation theorem, we directly deduce that
    \[
    \lim_{\rho\searrow 0} \frac{\mass\big(  \{ x' \in E_{k_0} \colon \rho(u(z),u(x'))\le \epsilon \} \cap \mass(B_{\rho}(z)\big)}{\mass(B_{\rho}(z))} =1, \qquad \mass\text{-a.e.\ } z\in E_{k_0}.
    \]
    Thus, by definition, $\mass$-a.e.\ point in $E_{k_0}$ is of approximate continuity for $u$. In particular, there is an approximate value $y = \tilde u(x)$ for suitable $x \in E_{k_0}$ such that $\essinf_{E_{k_0}} \rho(u(\,\cdot\,),y)=0$. Now we can continue by following the proof of \cite[Theorem 3.2]{LTpointwise}: in the arguments, we use for this setting \eqref{eq:doubling}-\eqref{eq:maximal operator BV} instead of the corresponding global versions and carry a constant $C(R)$ which might depend here on ${\sf Doub}(R)$.
\medskip\\\textbf{Implication} iii) $\Rightarrow$ i). We need to show that, for any $1$-Lipschitz function $\varphi \in \LIP(\YY)$ with $\varphi(\bar y)=0$, $\varphi \circ u \in \BV(\XX)$ and $|\DIFF(\varphi \circ u)| \le C \nu $. This follows by a standard application of a \emph{discrete convolution} technique (see, e.g., the proof of \cite[Proposition 4.1]{KinnunenKorteShanmugaligam14} for BV-functions) which, in the present note, can be performed building a partition of unity subjected to the covering of Lemma \ref{lem:covering}. We omit the details and instead notice the only needed straightforward modification with respect to the classical argument: while carrying on with the usual proof, to bound the difference of average values of $\varphi \circ u$ on two close-by balls $B_i,B_j$ of the covering, use Lemma \ref{threedottwo} for $c = \varphi(y_i)$ for the point $y_i \in \YY$ given by assumption iii) on $B_i$. Then, estimate $|\varphi \circ u-\varphi(y_i)|\le \rho(u,y_i)$ and use the assumption \eqref{fakepoincare}.
\medskip\\\textbf{Last statement}. From the first implication, we see that whenever $u \in \BV(\XX,\YY_{\bar y})$, then \eqref{eq:maximal estimate} holds with $\mu = |\DIFF u|$, $\sigma = 2\tau_p$ and $C(R)$ a constant depending only on the PI-parameters. 
	
	From the second implication, we therefore get that  $u \in \BV(\XX,\YY_{\bar y})$ implies that $\nu = |\DIFF u|$ is admissible, with $\lambda = 3\sigma = 6\tau_p$ and $C(R)$ again with the same dependence as in ii). 
    \end{proof}
    \begin{rem}
       {
       Let us comment on property iii). If $\YY$ is a Banach space, one of the equivalent definitions of a BV-function considered in \cite{MirandaJr03} is the following:
            the Poincar\'e inequality  as stated in \eqref{fakepoincare} holds with $y=\dashint_{B_r(x)}u\,\dd\mass$.
        We thus obtain the compatibility between the definitions of the space $\BV$ considered in this note and the notions of the space $\BV$ in  \cite{MirandaJr03}, for Banach targets.
       }\fr  
        \end{rem}
\begin{cor}[Lusin density of Lipschitz maps]
Let \((\XX,\dist,\mass)\) be a PI space and let \((\YY,\rho,\bar y)\) be a pointed separable metric space. Let \(u\in\BV(\XX,\YY_{\bar y})\) be given.
Then for every \(B\subseteq\XX\) bounded Borel and \(\varepsilon>0\) there exists \(E_\varepsilon\subseteq B\) Borel
s.t. \(\mass(B\setminus E_\varepsilon)\leq\varepsilon\) and \(u|_{E_\varepsilon}\colon E_\varepsilon\to\YY\) is Lipschitz.
\end{cor}
\begin{proof}
Fix any \(p\in\XX\) and pick \(R>0\) so that \(B\subseteq B_R(p)\). By virtue of Proposition \ref{thm:Equivalent BVmaps PI} ii),
\begin{equation}\label{eq:Lusin-Lip_aux}
\frac{\rho(u(x),u(y))}{\dist(x,y)}\leq C(2R)\big(\mathcal M_{2R\sigma}(|\DIFF u|)(x)+\mathcal M_{2R\sigma}(|\DIFF u|)(y)\big)
\quad \begin{array}{c}\forall x, y\in B_R(p)\setminus N,\\
x\neq y,
\end{array}
\end{equation}
for some constant \(\sigma\geq 1\) and some Borel set \(N\subseteq B_R(p)\) with \(\mass(N)=0\). Choosing \(\lambda>0\)
so that { $C(R,\sigma)\frac{|\DIFF u|(\XX)}{\lambda}\leq\varepsilon$ for a suitable $C(R,\sigma)>0$}, we know from \eqref{eq:maximal operator BV} that
\(E_\varepsilon\coloneqq\big\{x\in B_R(p)\setminus N\,\big|\,\mathcal M_{2R\sigma}(|\DIFF u|)(x)\leq\lambda\big\}\) satisfies
\(\mass(B\setminus E_\varepsilon)\leq\mass(B_R(p)\setminus E_\varepsilon)\leq\varepsilon\). Moreover, we deduce
from \eqref{eq:Lusin-Lip_aux} that
\[
\frac{\rho(u(x),u(y))}{\dist(x,y)}\leq 2 C(2R)\lambda\quad\text{ for every }x,y\in E_\varepsilon,
\]
which shows that \(u|_{E_\varepsilon}\) is \(2 C(2R)\lambda\)-Lipschitz. Consequently, the statement is achieved.
\end{proof}

We now start proving the required equivalence results to ultimately prove our main result Theorem \ref{thm:main result 2}.

    The following result follows at once from the lower semicontinuity of the total variation of real valued functions of bounded variation.
    \begin{prop}
    Let $(\XX,\dist,\mass)$ be a metric measure space and let $(\YY,\rho,\bar y)$ be a pointed separable metric space. Then, $\BV^*(\XX,\YY_{\bar y})\subseteq\BV(\XX,\YY_{\bar y})$ and it holds as measures
    \[
         |\DIFF u|\le V_u\qquad\text{for every }u\in \BV^*(\XX,\YY_{\bar y}).
    \]
    \end{prop}
\begin{thm}
	Let $(\XX,\dist,\mass)$ be a PI space and let $(\YY,\rho,\bar y)$ be a pointed separable metric space. 
	Then, $ \BV(\XX,\YY_{\bar y})\subseteq \BV^*(\XX,\YY_{\bar y})$ and
	$$
	  V_u\le C |\DIFF u|\qquad\text{for every }u\in \BV(\XX,\YY_{\bar y}),
	$$
 as measures, 
	for some constant $C>0$ depending only on the PI-parameters.
\end{thm}
\begin{proof}
It is enough to fix an arbitrary $\Omega\subseteq\XX$ open and build a sequence $(u_h)_h\subseteq \mathcal S_{\rm loc}(\Omega,\YY)$ with $f_h\rightarrow f$ in $L^1_{\rm loc}(\Omega,\YY)$ and
$$
\liminf_{h \nearrow \infty} |\DIFF u_h|(\Omega)\le C |\DIFF u|(\Omega)
$$
for some constant $C$ depending only on the PI-parameters. In what follows, possibly bigger constants $C>0$ with the same dependence might appear without being relabeled. We fix $h\ge 3$ and consider the Whitney type covering $\{B_i\}$ of $\Omega$, given by Lemma \ref{lem:covering} with $\Lambda=15\lambda\vee 1 $ for $\lambda= 6\tau_p$ (the reason for this choice will be clear in the arguments), to construct a competitor $u_h$. By Proposition \ref{thm:Equivalent BVmaps PI}, take, for every $i$, $y_i\in\YY$ such that 
\begin{equation}\label{csadmk}
	\int_{15B_i}\rho(u,y_i)\dd\mass\le C r_i|\DIFF u|( \Lambda B_i).
\end{equation}
We define $$u_h=\sum_i y_i \chi_{B_i\setminus \bigcup_{j<i} B_j},$$
where we recall that the covering $\{B_i\}_i$ depends on $h$, even though we do not make this dependence explicit.
Now we estimate, using  the fact that $\{B_i\}_i$ is a covering of $\Omega$ (property $2$ of the covering) for the first inequality, and \eqref{csadmk}  for the last inequality,
$$
\int_\Omega \rho(u,u_h)\dd\mass\le \sum_i \int_{B_i\setminus \bigcup_{j<i} B_j} \rho(u,y_i)\dd\mass\le \sum_i \int_{B_i} \rho(u,y_i)\dd\mass\le C  \sum_i r_i|\DIFF u|(\Lambda B_i).
$$
Recalling properties $1$ and $2$ of the covering we conclude that  
$$\int_\Omega \rho(u,u_h)\dd\mass\le h^{-1}C |\DIFF u|(\Omega).$$

Now we set $D_i\defeq B_i\setminus\bigcup_{j<i} B_j$, notice that $u_h=\sum_i y_i \chi_{D_i}$. We can use Lemma \ref{cdscascs}  and the fact that the covering $(B_i)_i$ is locally finite to deduce that $$|\DIFF u_h|\le C\sum_{i<j}\rho(y_i,y_j)\HH^h\mres(\partial^* D_i\cap\partial^* D_j).$$
As \begin{equation}\label{cscsda}
    \partial^* D_i\subseteq\bigcup_{p\le i}\partial^* B_p,
\end{equation}
we have \begin{equation}\label{csdsacs}
    |\DIFF u_h|\le C\sum_h \sum_{i<j}\rho(y_i,y_j)\HH^h\mres(\partial^* D_i\cap\partial^* D_j\cap\partial^* B_h).
\end{equation}
Now fix for  the moment $h$ and consider $H_{i,j,h}\defeq\partial^* D_i\cap\partial^* D_j\cap\partial^* B_h$.
We claim that $\HH^h(H_{i,j,h})=0$ unless $h=i<j$. Indeed, if $h>i$, $\HH^h(H_{i,j,h})=0$ by property $5$ of the covering and \eqref{cscsda}.
Now, if $h<i$, $\HH^h$-a.e.\ $x\in H_{i,j,h}$ satisfies $x\notin\partial B_i$. For such $x$, $x\notin \XX\setminus \bar B_i$, as $x\in\partial^* D_i$ but at the same time  $x\notin  B_i$, as $x\in\partial^* D_j$ ($j>i$), whence the claim.
Hence \eqref{csdsacs} reads 
$$
    |\DIFF u_h|\le C \sum_{i<j}\rho(y_i,y_j)\HH^h\mres(\partial^* D_i\cap\partial^* D_j\cap\partial^* B_i).
$$
Now, if $\bar B_i\cap\bar B_j=\varnothing$, then $\partial^* D_i\cap\partial^* D_j=\varnothing$, hence we can write 
\begin{equation}
\begin{aligned}
        |\DIFF u_h|&\le C \sum_{i}\sum_{j:\bar B_i\cap \bar B_j\ne\varnothing}\rho(y_i,y_j)\HH^h\mres(\partial^* D_i\cap\partial^* D_j\cap\partial^* B_i),
        \\&\le C\sum_{i}\sum_{j:\bar B_i\cap \bar B_j\ne\varnothing}\rho(y_i,y_j)\HH^h\mres \partial^* B_i
        \le C\sum_{i}\sum_{j:\bar B_i\cap \bar B_j\ne\varnothing}\rho(y_i,y_j)\per(B_i,\XX)\\&\le  C\sum_{i}\sum_{j:\bar B_i\cap \bar B_j\ne\varnothing}\rho(y_i,y_j)\frac{\mass(B_i)}{r_i}, 
\end{aligned} \label{eq:stim uh}
\end{equation}
where the last inequality is due to property $4$ of the covering. 
Now, consider $B_i$ and $B_j$ such that $\bar B_i\cap \bar B_j\ne \varnothing$. By property $3$ of the covering, $r_i\le 4 r_j$ and $r_j\le 4 r_i$, notice also that $\dist(x_i,x_j)\le r_i+r_j$. Then, we can invoke Lemma \ref{threedottwo} with $\delta/3 = r_i+r_j, r=r_j,s=r_i$, $d = 15$  to deduce
\begin{equation}\label{csanjocsno}
\begin{aligned}
\abs{\dashint_{B_j}\rho(u,y_j)\dd\mass-\dashint_{B_i}\rho(u,y_j)\dd\mass}&\le  \frac{C}{\mass(B_j)}   \int_{B_{3(r_i+r_j)}(x_j)}  \rho(u,y_j)\,\dd\mass, \\
&\le   \frac{C r_j}{\mass(B_j)}   \int_{15B_j}  \rho(u,y_j)\,\dd\mass  \\
&\overset{\eqref{csadmk}}{\le} \frac{C r_j}{\mass(B_j)}|\DIFF u|(\Lambda B_j).
\end{aligned}
\end{equation} 
Now we estimate, for $\bar B_i\cap \bar B_j\ne\varnothing$, using \eqref{csadmk}  and \eqref{csanjocsno}, as following
\begin{align*}
\rho(y_i,y_j)&=\dashint_{B_i}\rho(y_i,y_j)\dd\mass\\&\le \dashint_{B_i}\rho(y_i,u)\dd\mass+\dashint_{B_i}\rho(u,y_j)\dd\mass - \dashint_{B_j}\rho(u,y_j)\dd\mass+\dashint_{B_j}\rho(u,y_j)\dd\mass
\\&\le C \frac{r_i}{\mass(B_i)}|\DIFF u|(\Lambda B_i) +C \frac{r_j}{\mass(B_j)}\Big(|\DIFF u|(\Lambda B_i)+|\DIFF u|(\Lambda B_j) \Big) \\
&\le  C \frac{r_i}{\mass(B_i)} \Big(|\DIFF u|(\Lambda B_i) + |\DIFF u|(\Lambda B_j)\Big) 
\end{align*}
having used, in the last inequality, the fact that $r_j \le 4r_i, r_i\le 4_j$ and a doubling estimate. Finally, from \eqref{eq:stim uh}, we reach
\begin{align*}
	|\DIFF u_h|(\Omega)&\le C\sum_{i}\sum_{j:\bar B_i\cap \bar B_j\ne\varnothing} \Big(|\DIFF u|(\Lambda B_i) + |\DIFF u|(\Lambda B_j)\Big)
	\\&\le C\Big(\sum_i \sum_{j:\bar B_i\cap \bar B_j\ne\varnothing} |\DIFF u|(\Lambda  B_i)+ \sum_j \sum_{i:\bar B_i\cap\bar  B_j\ne\varnothing} |\DIFF u|( \Lambda B_j)\Big) \le C|\DIFF f|(\Omega),
\end{align*}
where we used property $2$  of the covering. 
\end{proof}
\begin{prop}
     Let $(\XX,\dist,\mass)$ be a metric measure space and let   $(\YY,\rho,\bar y)$ be a pointed separable metric space.  
     Then, $\BV_w(\XX,\YY_{\bar y}) \subseteq \BV(\XX,\YY_{\bar y})$ and it holds as measures
     $$
     |\DIFF u|\le |\DIFF u|_w\qquad\text{for every }u\in \BV_w(\XX,\YY_{\bar y}).
     $$
\end{prop}
\begin{proof}
If $u \in \BV_w(\XX,\YY_{\bar y})$, then by definition for every $\infty$-test plan $\pi$ and for every $1$-Lipschitz function $\varphi:\YY\rightarrow\RR$ with $\varphi(\bar y)=0$, we have (recalling also \eqref{eq:BV curve postcomp}) that $\varphi \circ u \circ \gamma \in \BV([0,1])$ for $\pi$-a.e.\ $\gamma$ and  we can estimate
	\[
	    \int \gamma_* |\DIFF ( \varphi \circ u \circ \gamma  )|\,\dd\pi \le \int \gamma_* |\DIFF (  u \circ \gamma  )|\,\dd\pi \le  {\rm Comp}(\pi){\rm Lip}(\pi)|\DIFF u|_w.
	\]
   By the equivalent characterization of \cite{AmbrosioDiMarino14}  we get $\varphi \circ u \in \BV(\XX)$ and $|\DIFF (\varphi \circ u)|\le |\DIFF u|_w$. By the arbitrariness of $\varphi$ we conclude.
\end{proof}
\begin{thm}
    Let $(\XX,\dist,\mass)$ be a PI space and let $(\YY,\rho,\bar y)$ be a pointed separable metric space. Then  $\BV^*(\XX,\YY_{\bar y}) \subseteq \BV_w(\XX,\YY_{\bar y}) $ and
    \[
       |\DIFF u|_w \le C V_u\qquad\text{for every }u \in  \BV^*(\XX,\YY_{\bar y}),
    \]
    as measures,
    for some constant $C>0$ depending only on the PI-parameters.
\end{thm}
\begin{proof}	
We subdivide the proof into two steps.
\medskip\\\textbf{Reduction to simple maps}. We  claim that it is enough to prove for all $u \in \SS(\XX,\YY) \cap \BV(\XX,\YY_{\bar y})$ that
\begin{equation}
		\int \gamma_*|\DIFF (u\circ \gamma)|\dd\pi(\gamma)\le {\mathrm {Comp}}(\pi) {\mathrm {Lip}}(\pi) C|\DIFF u|,\label{csdcsa}
\end{equation}
for every $\infty$-test plan and for a suitable constant $C>0$ depending only on the PI-parameters. Indeed, let $u \in \BV^*(\XX,\YY_{\bar y})$ be arbitrary and, by definition of $V_u$, consider $(u_n)\subseteq \SS_{\rm loc}(\XX,\YY)$ be optimal for the definition of $V_u(\XX)$, i.e.\ satisfying
\[
\rho(u_n,u)\to 0 \text{ in }L^1_{\rm loc}(\XX),\qquad \text{and}\qquad |\DIFF u_n| \to V_u,
\]
in duality with continuous and bounded functions. Let $\pi$ be an arbitrary \emph{bounded} $\infty$-test plan so that, denoting by $K$ the closure of $\{\gamma_t \colon \gamma \in \text{supp}(\pi),t\in[0,1]\} $, we have $K$ is bounded, hence compact. Since $u_n \in \BV_{\rm loc}(\XX,\YY)$ eventually for $n$ large, we can consider, via a cut-off argument as in Proposition \ref{prop:cut off}, maps $w_n \in \BV(\XX,\YY_{\bar y})$ agreeing with $u_n$ on a common bounded open neighborhood $\Omega \supset K$ for every $n \in\NN$ (discarding countably many negligible sets in $(0,r_0)$ and picking a common $\Omega_r$ in Proposition \ref{prop:cut off}). This guarantees also that $w_n \in \SS(\XX,\YY)$, as obtained by a cut-off with constant extension of a locally simple map in $\XX$. Thus,  by the assumption of this step, \eqref{csdcsa} can be written for each $w_n$ so that, by locality on $\Omega$ ($\pi$ is supported on curves living in $\Omega$)
\[
\iint \varphi \circ \gamma_t \, \dd |\DIFF (u_n\circ \gamma)|(t)\dd\pi(\gamma)\le {\mathrm {Comp}}(\pi) {\mathrm {Lip}}(\pi) C\int \varphi \, \dd |\DIFF u_n|,
\]
for every $\varphi \in C_b(\XX)$ and $n \in\NN$. Since $V_u(\XX)<\infty$ and $\rho(u_n,u) \to 0$ in $L^1(\Omega)$, by lower-semicontinuity \eqref{eq:BV curve lsc}, Fatou's lemma and arbitrariness of $\varphi\ge 0$, we get for $\pi$-a.e.\ $\gamma$ that $u\circ \gamma \in \BV([0,1],\YY)$  and 
\[
\int \gamma_* |\DIFF (u\circ \gamma)|\dd\pi(\gamma)\le {\mathrm {Comp}}(\pi) {\mathrm {Lip}}(\pi) C V_u.
\]
Recalling Remark \ref{rem:weak BV on bounded plans}, the above then holds for every $\infty$-test plan $\pi$ giving the conclusion.
\medskip\\\textbf{Proof for simple maps}. We thus only need to prove that for every $u \in \SS(\XX,\YY) \cap \BV(\XX,\YY_{\bar y})$, \eqref{csdcsa}
holds for every $\infty$-test plan $\pi$. Let $u= \sum_{i=1}^m\alpha_i\nchi_{E_i}$ for $(\alpha_i)\subseteq\YY$ and $(E_i)$ partition and $\pi$ be given. Notice that, for $\pi$-a.e.\ $\gamma$, we have that $u\circ \gamma$ is a simple curve taking values in $(\alpha_i)_i\subseteq\YY$ and, since $\rho(u, \alpha_i) - \rho(\alpha_i,\bar y) \in \BV(\XX) $, by the equivalent notion via test plan \cite{AmbrosioDiMarino14},  we have $\rho(u, \alpha_i)\circ \gamma \in \BV([0,1])$. In particular,
$$|\DIFF (u\circ \gamma)|\le |\DIFF\rho(u\circ\gamma(\,\cdot\,),\alpha_1)|\vee\dots\vee|\DIFF\rho(u\circ\gamma(\,\cdot\,),\alpha_m)|$$
so that $|\DIFF u|_w$ is a finite measure. We claim that the set $$
	\{(\gamma,t)\in  C([0,1],\XX) \times (0,1):t\in J_{u\circ\gamma}\}
	$$ is $(\pi\otimes\mathcal L^1)$-measurable. In order to show the claim, we can assume with no loss of generality that $u$ is real valued, as
 $$
 J_{u\circ \gamma}=\bigcup_\varphi J_{\varphi\circ u\circ \gamma},
 $$
 where the union is taken for $\varphi$ in a suitable  subfamily of $\LIP(\YY)$.
 Notice that it is enough to prove that for every $r\in(0,1)$, the map 
	 \begin{equation}\label{eqqqq1}
	 C([0,1],\XX) \times(0,1)\rightarrow [0,+\infty]\qquad\text{defined as}\qquad(\gamma,t)\mapsto |\DIFF(u\circ \gamma)|((t-r)\vee0,(t+r)\wedge 1)
	\end{equation}
	is $(\pi\otimes\mathcal L^1)$-measurable. 
 Consider first the map 
 \begin{equation}\label{eqqqq2}
     	 C([0,1],\XX)\rightarrow L^1((0,1),\YY)\qquad\text{defined as}\qquad\gamma\mapsto u\circ \gamma.
 \end{equation}
This map  is measurable, as it can be easily checked by approximation of $u$ with Lipschitz functions (and for Lipschitz functions, such map is continous).
Now, if   $\Gamma_n\subseteq C([0,1],\XX)$ is a set on which the map in \eqref{eqqqq2} is continuous, the map in \eqref{eqqqq1} is lower semicontinuous on $\Gamma_n\times (0,1)$, so that the claim is proved, taking into account Lusin's Theorem.

	By the discussion above and by Fubini's Theorem, for $\mathcal{L}^1$-a.e.\ $t\in (0,1)$, $t \notin J_{u\circ \gamma}$ for $\pi$-a.e.\ $\gamma$. Hence for every $\epsilon\in(0,1)$ we can find a partition $\mathcal{P}^\epsilon=\{0=t^\epsilon_0<t^\epsilon_1<\dots <t^\epsilon_{n^\epsilon}=1\}$ such that for a set $G^\epsilon$ with $\pi(C([0,1],\XX) \setminus G^\epsilon)<\epsilon$, for every curve $\gamma\in G^\epsilon$, $J_{u\circ\gamma}\cap\{t^\epsilon_0,t_i^\epsilon,\dots,t_{n^\epsilon}^\epsilon\}=\varnothing$ and $\#J_{u\circ\gamma}\cap [t^\epsilon_i,t^\epsilon_{i+1}]\le 1$ for every $i=0,\dots,n^\epsilon-1$.	
	Now, for $\epsilon\in(0,1)$ and $i=0,\dots,n^\epsilon-1$  we define 
	$$
	\pi^\epsilon_i\defeq \frac{1}{\pi(G^\epsilon)}\big({{\mathrm {restr}}}_{t^\epsilon_i}^{t^\epsilon_{i+1}}\big)_* (\pi\mres G^\epsilon).
	$$
	Define also, for ${j,k}\in\{1,\dots,m\}$,
	$$G_{j,k}\defeq\Big\{\gamma \in C([0,1],\XX) :{\mathrm{ap \, lim}}_{t\searrow  0} u\circ\gamma(t)=\alpha_j, \mathrm{ap\, lim}_{t\nearrow 1} u\circ\gamma(t)=\alpha_k\Big\},$$
	(the approximate limits are well defined being the map $u\circ \gamma$ a simple curve of bounded variation) and (when $\pi_i^\epsilon(G_{j,k})=0$ we understand $\pi^\epsilon_{i,j,k}=0$) $$\pi_{i,j,k}^\epsilon\defeq \frac{1}{\pi_i^\epsilon (G_{j,k})} (\pi_i^\epsilon\mres G_{j,k}).$$
	For $\epsilon\in(0,1)$, $i=0,\dots,n^\epsilon-1$ and ${j,k}\in\{1,\dots,m\}$, define
	$u_{j,k}\defeq \alpha_j\chi_{E_j}+\alpha_k\chi_{\XX\setminus E_j}$. Notice that for $\pi^\epsilon_{i,j,k}$-a.e.\ $\gamma$, $u\circ\gamma=u_{j,k}\circ \gamma$ a.e.
	Then, by the equivalent characterization of $\BV(\XX)$ of \cite{AmbrosioDiMarino14}, we deduce (for $C>0$ depending only on the PI-parameters given by Lemma \ref{cdscascs})
	\begin{align*}
		\int \gamma_* |\DIFF(u\circ\gamma)| \dd {\pi_{i,j,k}^\epsilon}(\gamma)&=\int \gamma_* |\DIFF(u_{j,k}\circ\gamma)| \dd {\pi_{i,j,k}^\epsilon}(\gamma)=\int \gamma_* |\DIFF\rho(u_{j,k}\circ\gamma(\,\cdot\,),\alpha_j)| \dd {\pi_{i,j,k}^\epsilon}(\gamma)\\&\le {\mathrm {Comp}}({\pi_{i,j,k}^\epsilon}) {\mathrm {LIP}}({\pi_{i,j,k}^\epsilon})|\DIFF u_{j,k}|
		\\&\le \frac{C}{\pi_i^\epsilon(G_{j,k})\pi(G^\epsilon)}{\mathrm {Comp}}({\pi}) {\mathrm (t_{i+1}^\epsilon-t_i^\epsilon){\mathrm{LIP}}}(\pi) \rho(\alpha_j,\alpha_k)\HH^h\mres \partial^*{E_j}.
	\end{align*}
Similarly,
$$
\int \gamma_* |\DIFF(u\circ\gamma)| \dd {\pi_{i,j,k}^\epsilon}(\gamma)\le \frac{C}{\pi^\epsilon(G_{j,k})\pi(G^\epsilon)}{\mathrm {Comp}}({\pi}) {\mathrm (t_{i+1}^\epsilon-t_i^\epsilon){\mathrm{LIP}}}(\pi) \rho(\alpha_j,\alpha_k)\HH^h\mres \partial^*{E_k},
$$
so that
$$
\int \gamma_* |\DIFF(u\circ\gamma)| \dd {\pi_{i,j,k}^\epsilon}(\gamma)\le \frac{C\,\mathrm{LIP}(\pi)}{\pi_i^\epsilon(G_{j,k})\pi(G^\epsilon)}{\mathrm {Comp}}({\pi}) \mathrm (t_{i+1}^\epsilon-t_i^\epsilon)\rho(\alpha_j,\alpha_k)\HH^h\mres (\partial^*{E_j}\cap \partial^*E_k).
$$
Therefore, for every $\epsilon\in (0,1)$ and $i=0,\dots, n^\epsilon-1$, 
\begin{align*}
	&\int\gamma_* |\DIFF(u\circ\gamma)| \dd({{\mathrm {restr}}}[t^\epsilon_i,t^\epsilon_{i+1}])_*(\pi\mres G^\epsilon)(\gamma) =\pi(G^\epsilon)\int\gamma_* |\DIFF(u\circ\gamma)| \dd\pi^\epsilon_i(\gamma)\\&\qquad={\pi (G^\epsilon)}\sum_{j,k}{\pi_i^\epsilon (G^\epsilon_{i,j,k})}\int \gamma_* |\DIFF(u\circ\gamma)| \dd\pi^\epsilon_{i,j,k}(\gamma)
	\\&\qquad\le C\sum_{j,k} {\mathrm {Comp}}({\pi}) {\mathrm (t_{i+1}^\epsilon-t_i^\epsilon){\mathrm{LIP}}}(\pi) \rho(\alpha_j,\alpha_k)\HH^h\mres (\partial^*{E_j}\cap\partial^* E_k).
\end{align*}
Now, summing over $i=0,\dots,n^\epsilon-1$, we reach 
\begin{align*}
	\int_{G^\epsilon}\gamma_* |\DIFF(u\circ\gamma)| \dd \pi(\gamma)&\le C {\mathrm {Comp}}({\pi}) {\mathrm{LIP}}(\pi) \sum_{j,k}\rho(\alpha_j,\alpha_k)\HH^h\mres (\partial^*{E_j}\cap\partial^* E_k)\\&\le {\mathrm {Comp}}({\pi}) {\mathrm{LIP}}(\pi)  C|\DIFF u|,
\end{align*}
having used, in the last inequality, Lemma \ref{cdscascs}. Then \eqref{csdcsa} follows letting $\epsilon\searrow 0$ and the proof is concluded. 
\end{proof}
%
%
%
%
%
%

%
%
\end{document}